\documentclass{article}
\usepackage[title]{appendix}
\usepackage{epsfig, amsmath, amssymb, amsthm}
\usepackage{color}
\usepackage{esint}

\usepackage{tikz}
\usetikzlibrary{trees}
\usepackage{tkz-euclide}

\newcommand{\field}[1]{\mathbb{#1}}

\newcommand{\N}{\field{N}}    
\newcommand{\R}{\field{R}}    
\newcommand{\E}{\field{E}}      
\newcommand{\fP}{\field{P}}     
\newcommand{\T}{\field{T}}      

\newcommand{\cB}{{\mathcal B}}  
\newcommand{\cD}{{\mathcal D}}  
\newcommand{\cE}{{\mathcal E}}  
\newcommand{\cF}{{\mathcal F}}  
\newcommand{\cG}{{\mathcal G}}  
\newcommand{\cH}{{\mathcal H}}  
\newcommand{\cK}{{\mathcal K}}  
\newcommand{\cL}{{\mathcal L}}  
\newcommand{\cN}{{\mathcal N}}  
\newcommand{\cO}{{\mathcal O}}  

\newcommand{\bA}{\mathbf{A}}

\newcommand{\1}{\mathbf{1}}  

\newcommand{\eps}{\varepsilon}
\def\be{\begin{equation}}
  \def\ee{\end{equation}}
\def\lbl{\label}
\def\p{\partial}

\newcommand{\conE}[2]{\E\left(\left.#1\right| #2\right)}  

\DeclareMathOperator{\esssup}{ess\,sup}

\newtheorem{theorem}{Theorem}[section]
\newtheorem{proposition}[theorem]{Proposition}
\newtheorem{lemma}[theorem]{Lemma}
\newtheorem{cor}[theorem]{Corollary}
\newtheorem{remark}[theorem]{Remark}

\newtheorem{assumption}[theorem]{Assumption}

\title{From PDEs on standard domains to self-similar particle systems on fractals}
\date{\today}
\author{
  Georgi S. Medvedev\thanks{Department of Mathematics, Drexel University,
	 Philadelphia, PA 19104, USA, 
		{\tt medvedev@drexel.edu}}
              \and  Emmanuel Tr{\'e}lat \thanks{Sorbonne Universit\'e, Universit\'e Paris Cit\'e, CNRS, Inria, Laboratoire Jacques-Louis Lions, LJLL, F-75005 Paris, France,
               	{\tt emmanuel.trelat@sorbonne-universite.fr}
              }
              }

\date{}

\begin{document}
\maketitle

\begin{abstract}
We construct transported PDEs on self-similar fractal domains
from reference equations posed on the unit interval, and derive
explicit self-similar interacting particle systems that approximate
the resulting dynamics. The construction combines a
measure-preserving isometry between $L^2$-spaces on
$[0,1]$ and on the fractal \cite{Med2026}, a nonlocal-to-local
approximation of differential operators \cite{PauTre2025}, and a
Galerkin discretization on the canonical self-similar partitions.
This yields a two-parameter approximation scheme whose error
separates a nonlocal consistency term from a Galerkin network term.
We work out the transport, Burgers, and heat equations, discuss the
relation with intrinsic operators on fractals, and outline extensions
to local charts and to pullbacks of nonlocal equations on fractal
domains.

Moreover, the reverse mapping transforms  a nonlocal evolution
equation on a fractal domains into the evolution equation on the unit
interval, where the methods of classical numerical analysis can be
applied. This suggests a promising direction for the
development of numerical methods for nonlocal models on fractals,
including fractional heat equation, fractal scattering, and related models.
\end{abstract}

\section{Introduction}\label{sec.intro}

Partial differential equations on fractal domains arise naturally in models of transport,
diffusion, and wave propagation in geometrically irregular media. Classical examples
include diffusion in porous or ramified materials, signal propagation on hierarchical
structures, and scattering by fractal interfaces. The difficulty is that a fractal set
typically has no smooth atlas in the usual sense, so differential operators cannot be
introduced by direct geometric differentiation.

A large part of the existing theory therefore proceeds intrinsically, either through
Dirichlet forms and diffusion processes, or through renormalized graph approximations of
the Laplacian and related operators. For post-critically finite self-similar sets, this
viewpoint leads to Kigami's Laplacian and to a rich analytic theory of elliptic,
parabolic, and variational problems on fractals; see, for instance,
\cite{FukShi92,Kig01,Str06,HinTep13,HinMei20,Mosco13,CHT18}. Our objective here is
different. We do not attempt to recover an intrinsic differential calculus on the
fractal. Instead, we construct \emph{transported} PDEs on a self-similar set $K$ from
a reference PDE posed on a simpler domain $Q$, and we derive from this
construction explicit finite-dimensional interacting particle systems (IPS) adapted to the
self-similar structure of $K$.

The basic mechanism is simple. Recent work of \cite{Med2026} provides an
explicit measure-preserving isomorphism between a self-similar set $K$ endowed with a
stationary self-similar measure $\nu$ and an isomorphic iterated function system on the
unit interval $Q=[0,1]$, endowed with the corresponding stationary measure $\lambda$.
In the uniform case relevant for the concrete PDE examples below, $\lambda$ is simply
the Lebesgue measure on $[0,1]$. This isomorphism induces a linear isometry
\[
  T:L^2(K,\nu)\longrightarrow L^2(Q,\lambda),
\]
together with its inverse $T^{-1}$. Therefore, if $\bA$ is the generator of a
well-posed evolution equation on $Q$, then the conjugated operator
$T^{-1}\bA T$ defines a corresponding evolution on $K$.

\begin{figure}
 	\centering
 \includegraphics[width =.45\textwidth]{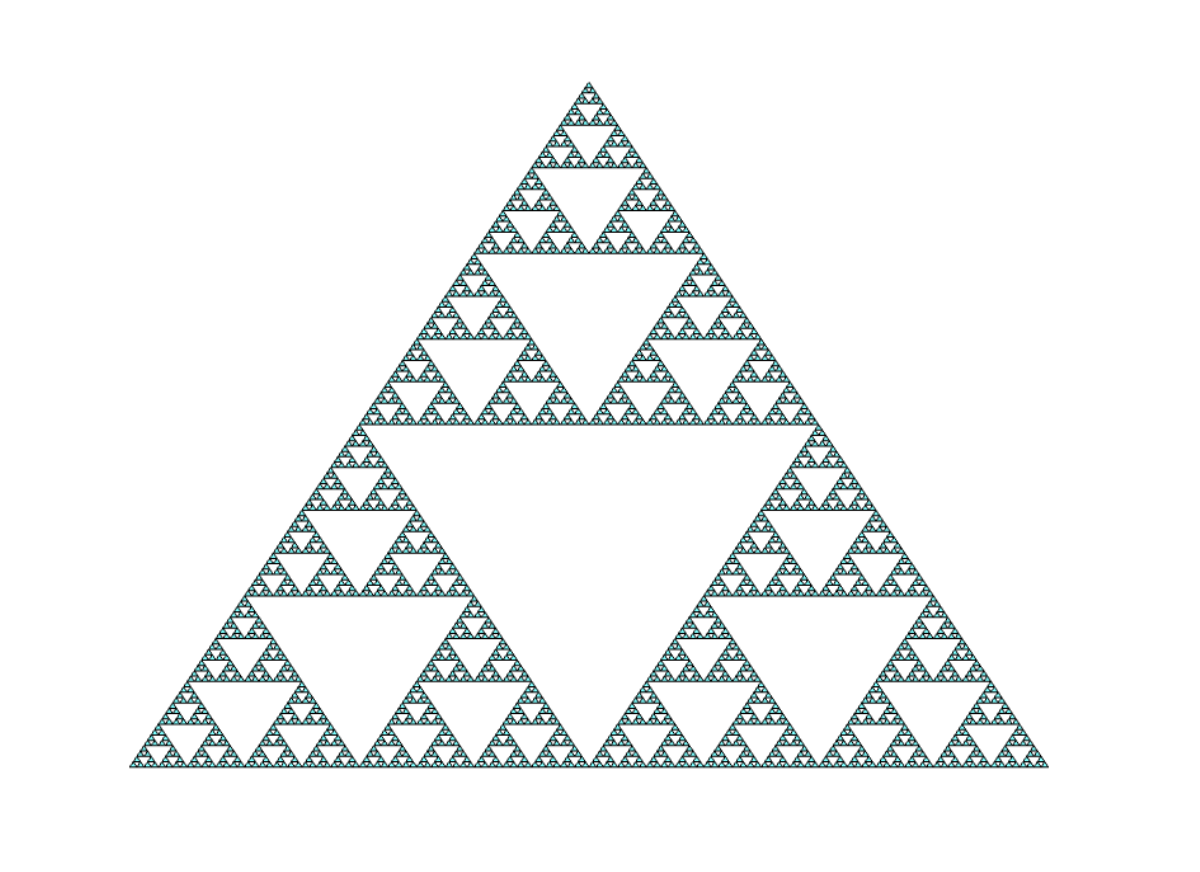}
  \caption{The Sierpinski gasket, which serves as a prototype throughout the paper.}
  \label{f.SG}
 \end{figure}

This operator-theoretic lifting, however, is only the first part of the story. In order
to connect the transported PDE to explicit particle systems, we insert an intermediate
nonlocal regularization step, following the nonlocal-to-local approximation framework of
\cite{PauTre2025}. Thus, starting from a PDE on $Q$, we first approximate its generator
by a family of integral operators $\bA_\eps$, then transport these operators to $K$, and
finally discretize the lifted nonlocal equation on the canonical level-$m$ partition of
the self-similar set. The resulting ODE system is a self-similar interacting particle
system whose nodes are the cells $K_w$ of generation $m$.

The paper has three main objectives. First, we formulate an abstract two-scale framework
that combines the nonlocal-to-local approximation on $Q$ with the Galerkin/IPS
approximation on $K$. Second, we turn this framework into explicit constructions for the
transport, Burgers, and heat equations, obtaining concrete self-similar particle systems
and quantitative error estimates. Third, we clarify the scope of the construction: the
operators transported to $K$ are not, in general, intrinsic operators in the sense of
fractal analysis, but they are natural continuum limits of hierarchical networks built
directly from the underlying iterated function system (IFS).

This distinction between transported and intrinsic operators is conceptually important.
If one aims at the canonical Laplacian attached to the geometry of the fractal, then one
must retain the adjacency structure of prefractal graphs and the renormalized Dirichlet
form. The global coding map used here preserves measure and symbolic cylinders, but not
the adjacency relations between cells in Euclidean space. On the other hand, if the aim
is to model dynamics on self-similar networks, or to import a well-understood PDE from a
simple reference domain to a hierarchical one, then the present construction gives an
explicit and flexible alternative. We return to this issue in the final section, where
we also discuss two possible continuations of the present work: local charts on the
Sierpinski gasket, and the pullback of nonlocal fractal equations (for
instance fractional powers of an intrinsic Laplacian) to the reference interval.

Let us also emphasize the relation with the intrinsic first-order calculus on fractals.
The Dirichlet-form approach of \cite{CipSau03} and the subsequent vector-analysis
framework of \cite{HinTep13,HinRocTep13} construct gradients,
divergences and quasilinear PDEs from energy measures and resistance forms. 
Recent works on Burgers-type equations,
transport/continuity equations, graph approximations and hydrodynamic limits on
fractals further show that first-order and particle-system questions can be developed
intrinsically \cite{HinMei20,HinSche24,HinMei22,CHT18}. The present paper 
is complementary to that line of work: our operators are transported through a
symbolic measure isomorphism and are designed to yield explicit self-similar
Galerkin/IPS approximations, whereas the intrinsic theory is governed by energy,
harmonic coordinates, and the adjacency structure of prefractal graphs.

The paper is organized as follows. Section~\ref{sec.fractals} recalls the self-similar
and measure-theoretic background. Section~\ref{sec.isomorph} reviews the canonical
measure-preserving isomorphism between the fractal and the unit interval and the induced
isometries on function spaces. Section~\ref{sec.evolution} summarizes the abstract
nonlocal evolution problem on a self-similar set and its deterministic and random
Galerkin approximations. Section~\ref{sec.prop} contains the general lifting framework
from PDEs on $Q$ to nonlocal equations and self-similar IPS on $K$. Section~\ref{sec.examples}
is devoted to the transport, Burgers, and heat equations. Finally,
Section~\ref{sec.outlook} discusses extensions, limitations, and possible bridges with
intrinsic operators and native nonlocal equations on fractals.

\section{Self-similar sets, symbolic coding, and self-similar measures}\label{sec.fractals}
In this section, we collect the measure-theoretic background on self-similar sets and stationary self-similar measures that will be used throughout the paper. Standard references are \cite{BSS-self-similar,BP-Fractals,Hut81}.

\paragraph{Iterated function systems and attractors.}
Let $\mathcal{F}=\{f_i:\mathbb{R}^d \to \mathbb{R}^d \mid i\in [k]\}$
be a finite collection of \textit{similitudes}
\begin{equation*}
  |f_i(x)-f_i(y)| = r_i |x-y| 
  \quad \text{for all } x,y\in\mathbb{R}^d,\; i\in[k],
\end{equation*}
with $0<r_i<1$.
A unique compact set $K$ satisfying
$K = \bigcup_{i\in[k]} f_i(K)$
is called the \emph{attractor} of the IFS $\mathcal{F}=\{f_i\}_{i\in [k]}$ 
(cf.~\cite[Theorem~2.1.1]{BP-Fractals}). $K$ is also called a \emph{self-similar} set.

\paragraph{Symbolic space and Bernoulli measures.}
There is a natural family of self-similar measures associated with $K$. These measures are
most conveniently described through Bernoulli measures on the symbolic space
$\Sigma\doteq [k]^\N$ which we review first.

For each $n\ge1$, let $\Sigma_n$ be the set of words of length $n$,
and let $\Sigma^\ast=\bigcup_{n=0}^\infty\Sigma_n$
denote the collection of all finite words over $[k]$.

We equip $\Sigma$ with the metric $\rho(i,j)=k^{-|i\wedge j|}$, for $i=(i_1,i_2,\dots)$ and $j=(j_1,j_2,\dots)$,
where $i\wedge j$ stands for the common prefix of the two sequences $i$ and $j$, and
$|i\wedge j|$ denotes its
length. Thus, $|i\wedge j|=\sup\{l\in\N:\, i_l=j_l\}$ if $i\wedge j\neq\emptyset$ and $0$ otherwise.
This metric defines a topology on $\Sigma$ that is consistent with the product
topology. $(\Sigma,\rho)$ is a compact metric space (cf.~\cite{BSS-self-similar}).

Let $i_1,\dots,i_n\in[k]$ and  define the cylinders
\begin{equation*}
[i_1 i_2 \dots i_n] \doteq \{j=(j_1j_2\dots) \in\Sigma:\; j_l=i_l,\; 1\le l\le n\}.
\end{equation*}
The collection of all cylinders generates the Borel $\sigma$-algebra in $\Sigma$, $\cB(\Sigma)$.

We say that $p=(p_1,p_2,\dots,p_k)$ is \textit{a probability vector}, if
$p_i>0$ for all $i\in[k]$ and $\sum_{i=1}^k p_i=1$.
The Bernoulli measure corresponding to $p$ is first defined on cylinders 
$\mu_p([i_1 i_2 \dots i_n])=p_{i_1}p_{i_2}\cdots p_{i_n}$,
and is then extended uniquely to a Borel probability measure on $(\Sigma,\cB(\Sigma))$.

\paragraph{Natural projection and self-similar measures.}
The symbolic representation of $K$ is provided by the natural projection $\pi_K:\Sigma\to K$.

For $w=(w_1\ldots w_n)\in\Sigma_n$, we set
$f_w\doteq f_{w_1}\circ\cdots\circ f_{w_n}$ and $K_w\doteq f_w(K)$.
The natural projection is defined by
\begin{equation*}\label{def-pi}
  \pi_K(w)=\bigcap_{n\ge1} K_{w_1\ldots w_n}, \qquad w=(w_1w_2\dots)\in\Sigma.
\end{equation*}
Since $\{K_{w_1\ldots w_n}\}_{n\ge1}$ is a nested sequence of compact sets with diameters tending to zero,
their intersection consists of a single point. The map $\pi_K$ is continuous and surjective
(cf.~\cite[Theorem~1.2.3]{Kig01}). However, it may not be injective, i.e.,
\begin{equation*}
  \cN_{\pi_K}=\left\{w\in\Sigma:\; \operatorname{card}\left\{\pi_K^{-1}(\pi_K(w))\right\}>1\right\}.
\end{equation*}
may not be empty. Indeed, this happens for many interesting examples of self-similar sets such as SG.

Throughout this paper, we make the following standing assumption.
\begin{assumption}\label{as.small}
There holds $\mu_p(\cN_{\pi_K})=0$.
\end{assumption}

We are now ready to introduce the stationary self-similar measure
$\nu_p=(\pi_K)_\ast\mu_p$:
\begin{equation*}
\nu_p(A)=\mu_p(\pi_K^{-1}(A)), \; A\in\cB(K).
\end{equation*}
The pushforward measure $\nu_p=(\pi_K)_\ast\mu_p$ is \textit{the stationary measure} of
the probabilistic IFS $(K,\cF, \nu_p)$, which means that
\begin{equation*}
\nu_p(A)=\sum_{j=1}^k p_j\,\nu_p\bigl(f_j^{-1}(A)\bigr),
\qquad A\in\cB(K).
\end{equation*}
The stationary measure $\nu_p$ is unique \cite[Theorem~2.1.1]{BP-Fractals}.
It is called \emph{a self-similar measure}.

Under Assumption~\ref{as.small}, we have the following explicit formula for the
pushforward measure $\nu_p$ on the cylinders:
\begin{equation}\label{nu-cyl}
\nu_p(K_w)=p_{w_1}p_{w_2}\dots p_{w_n},
\qquad \forall w\in\Sigma_n,\; n\in\N.
\end{equation}

\section{Measure-preserving isomorphism with the unit interval}\label{sec.isomorph}
\setcounter{equation}{0}

The key tool for transporting PDEs from $Q$ to $K$ is the measure-preserving isomorphism between the function spaces $L^2(Q,\lambda_p)$ and $L^2(K,\nu_p)$. In this section we recall the construction from \cite{Med2026} and the induced transport of functions and kernels.

\subsection{The model IFS on the unit interval}\label{sec.iso-IFS}
Given a probabilistic IFS $(K, \{f_i\}_{i\in [k]}, \nu_p)$, we construct an isomorphic IFS on the unit interval $Q=[0,1]$, denoted $(Q, \{g_i\}_{i\in [k]}, \lambda_p)$.

To this end, note that
$Q$ is the attractor of the following IFS
$$
Q=\bigcup_{i=1}^k g_i(Q),
$$
where $g_i(x)=\frac{x}{k} + x_{i-1}$ for $i\in [k]$,
and $x_i=\frac{i}{k}$. Denote $\cG=\{g_i\}_{i\in [k]}$.

We denote the natural projection ${\pi_Q}:\Sigma\to Q$:
\begin{equation*}
\Sigma\ni w=(w_1w_2\dots) \mapsto \bigcap_{j=1}^\infty g_{w_1w_2\dots w_j} (Q).
\end{equation*}
Let $\lambda_p=({\pi_Q})_\ast\mu_p$ be the pushforward of the Bernoulli
measure. On cylinders $Q_{w_1w_2\dots w_n}=g_{w_1w_2\dots w_n}(Q)$, we have
\begin{equation*}
  \lambda_p(Q_{w_1w_2\dots w_n})=
  \mu_p\big( {\pi_Q}^{-1} (Q_{w_1w_2\dots w_n}) \big) =p_{w_1}p_{w_2}\dots p_{w_n}.
\end{equation*}
In particular, for the uniform $p=(k^{-1}, k^{-1},\dots, k^{-1})$,
$\lambda_p$ coincides with the Lebesgue measure.


For $(Q,\cG, \lambda_p)$, one can easily verify
\begin{equation}\label{NpiQ}
  \mu_p\left(\mathcal{N}_{\pi_Q}\right)=0.
  \end{equation}

\begin{theorem}\label{thm.isomorph}
Under Assumption~\ref{as.small}, there exist full-measure subsets $Q_0\subset Q$ 
and $K_0\subset K$ and a measure-preserving bijection $\Phi:Q_0\to K_0$ 
such that $  \Phi\circ g_i = f_i\circ \Phi$ on $Q_0\cap g_i^{-1}(Q_0)$, for $i\in[k]$.
Equivalently, the probabilistic IFS $(K,\mathcal F,\nu_p)$ and $(Q,\mathcal G,\lambda_p)$ are isomorphic.
\end{theorem}

This is a reformulation of \cite[Theorem~5.7]{Med2026}.

\begin{remark}\label{rem:IFS-freedom}
The correspondence $f_i\leftrightarrow g_i$ used in Theorem~\ref{thm.isomorph} is not unique: any permutation $\sigma\in\mathfrak{S}_k$ yields an equally valid isomorphism $\Phi_\sigma$ satisfying $\Phi_\sigma\circ g_i = f_{\sigma(i)}\circ\Phi_\sigma$. Since each choice preserves the measure structure and the self-similar partition, the transported operators constructed in the sequel depend on this choice only through the labeling of the cells $K_w$. For the error estimates in Section~\ref{sec.prop}, all such choices lead to the same convergence rates, so the particular convention adopted here (the identity permutation) entails no loss of generality.
\end{remark}

\subsection{Transport of function spaces}
Having constructed an isomorphic IFS \((Q,\mathcal{G},\lambda_p)\), we now transport the measure-space isomorphism to the function spaces \(L^1(K,\nu_p)\) and \(L^1(Q,\lambda_p)\).

By Assumption~\ref{as.small} and \eqref{NpiQ},  the image measures
$\nu=(\pi_K)_*\mu$ and $\lambda=(\pi_Q)_*\mu$
are measure-preserving\footnote{Since \(p\) is fixed, we henceforth omit
  it from the notation of the measures
\(\mu\), \(\nu\), and \(\lambda\).}.

Define the pullback operators $U_K:L^1(K,\nu)\to L^1(\Sigma,\mu)$ and
$U_Q:L^1(Q,\lambda)\to L^1(\Sigma,\mu)$
by
\[
U_K f = f\circ \pi_K,
\qquad
U_Q g = g\circ \pi_Q .
\]
Since \(\pi_K\) and \(\pi_Q\) are measure-preserving, both operators are linear
isometries onto their respective ranges. Here, the term
\emph{isometry} may be understood in the $L^1$- or $L^2$-sense,
namely as a linear map
preserving the corresponding norm. In the \(L^2\)-setting, the  isometries
$U_K$ and $U_Q$ preserve the inner product as well.
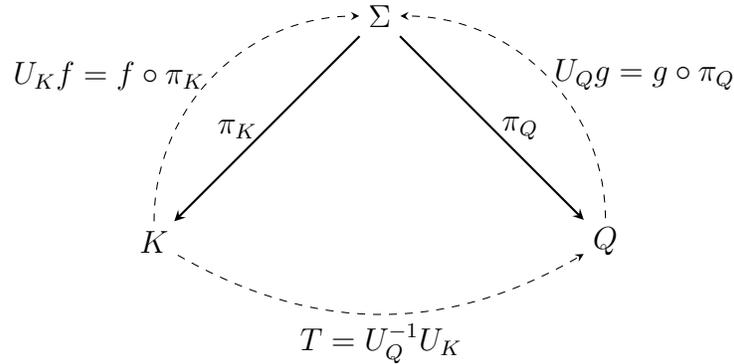
\begin{figure}[h]
  \begin{center}
    \begin{tikzpicture}[>=stealth, node distance=3cm, every node/.style={font=\large}]

\node (Sigma) at (0,3) {$\Sigma$};
\node (K) at (-3,0) {$K$};
\node (Q) at (3,0) {$Q$};

\draw[->, thick] (Sigma) -- (K) node[midway, left] {$\pi_K$};
\draw[->, thick] (Sigma) -- (Q) node[midway, right] {$\pi_Q$};

\draw[->, dashed, bend left=45] (K) to node[midway, left] {$U_Kf=f\circ\pi_K$} (Sigma);
\draw[->, dashed, bend right=45] (Q) to node[midway, right] {$U_Qg=g\circ\pi_Q$} (Sigma);
\draw[->, dashed, bend right=30] (K) to node[midway, below] {$T = U_Q^{-1}U_K$} (Q);
\end{tikzpicture}
   
\end{center}
\caption{The isomorphism between measure spaces $(K,\nu)$ and $(Q,\lambda)$ induces a
  corresponding isomorphism between the function spaces
  $L^1(K,\nu)$ and $L^1(Q,\lambda)$.}
\label{f.diagram}
\end{figure}

Moreover, since $\pi_Q$ is bijective $\mu$-a.e., the inverse of $U_Q$ is well-defined (up to null sets) on $L^1(\Sigma,\mu)$ by $U_Q^{-1}F= F\circ \pi_Q^{-1}$.
This allows us to define a canonical operator
\begin{equation}\label{def-T}
T:L^1(K,\nu)\longrightarrow L^1(Q,\lambda),
\qquad
Tf := U_Q^{-1}(U_K f),
\end{equation}
Equivalently, \(Tf\) is the unique (up to \(\lambda\)-null sets) function satisfying
\begin{equation}\label{Tf-conj}
f\circ\pi_K = (Tf)\circ\pi_Q
\quad\mu\text{-a.e.}.
\end{equation}

By construction, \(T\) is a surjective linear isometry preserving positivity and integrals, i.e.,
$\int_K f\,d\nu = \int_Q Tf\,d\lambda$ for every $f\in L^1(K,\nu)$.

\begin{figure}[h]
 	\centering
 \textbf{a}\includegraphics[width =.45\textwidth]{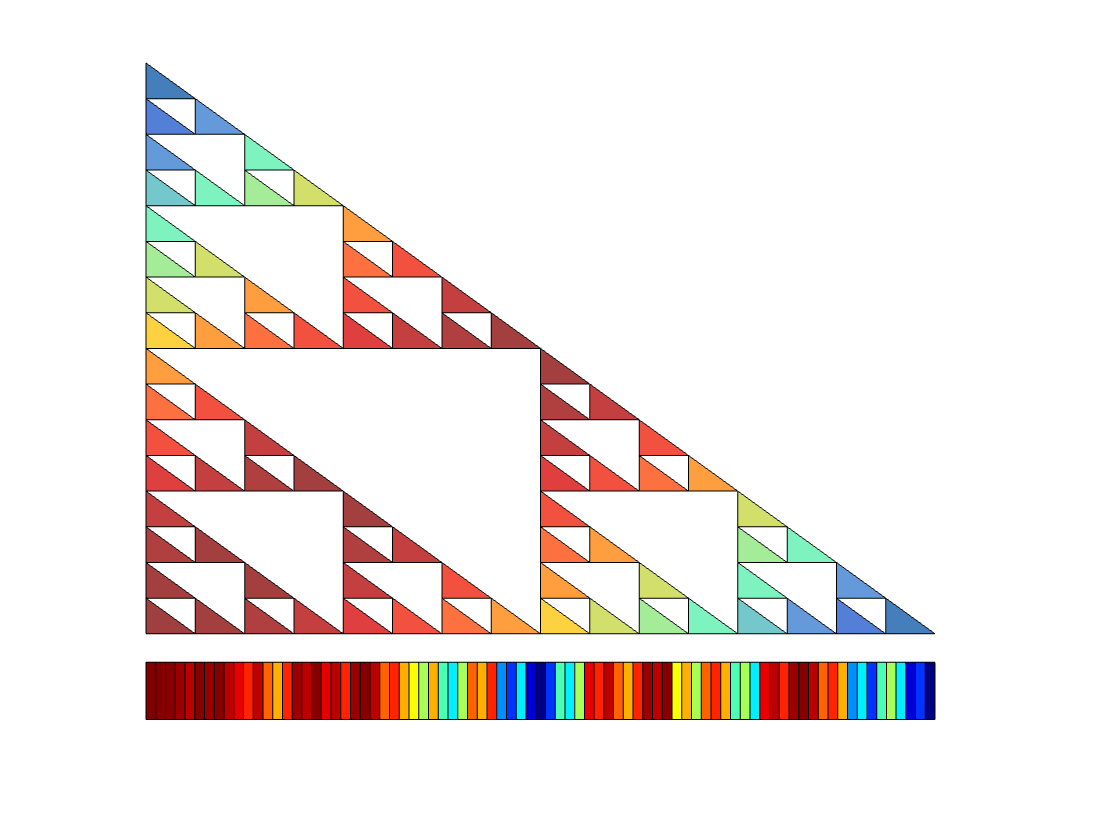}
 \textbf{b} \includegraphics[width =.45\textwidth]{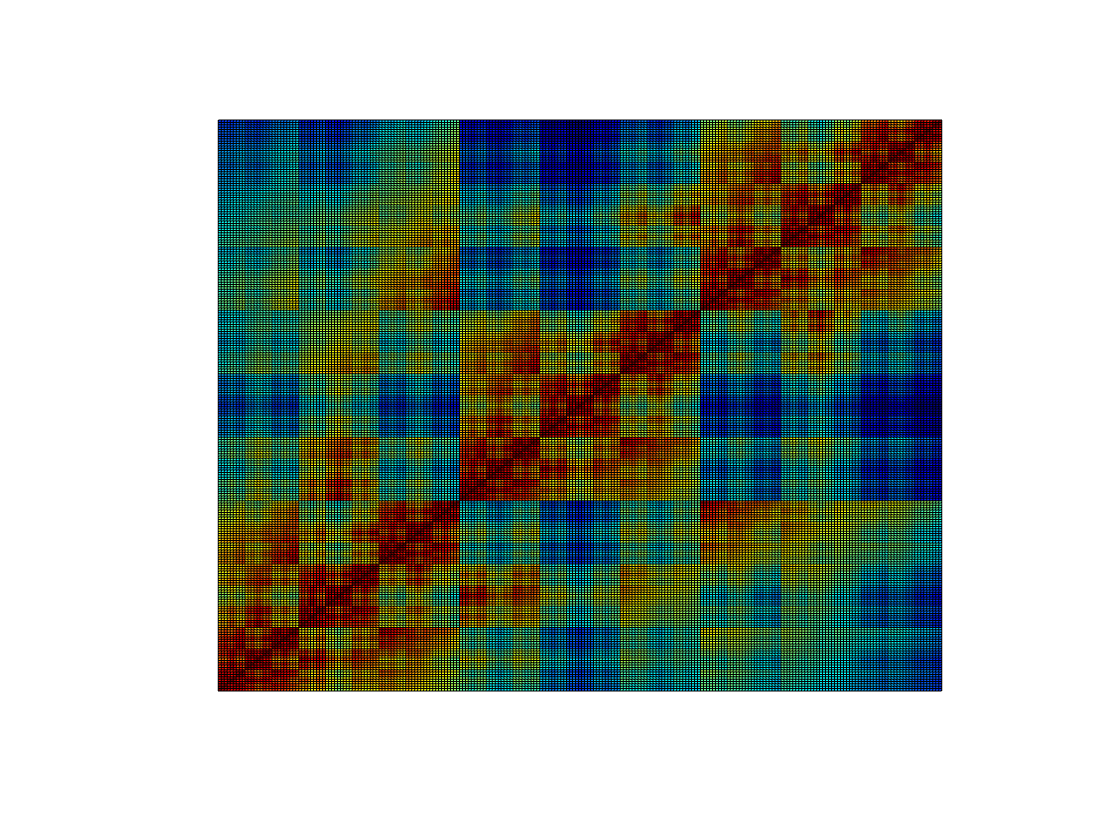}\\
  \caption{ \textbf{a}. The heat map of the function $f(x)=e^{-|x_1-x_2|}$ on SG
    $G\subset\R^2$ (above) and it is representation as a function on $[0,1]$ (below).
    \textbf{b}. Graphon representation of $W(x,y)=e^{-\|x-y\|}$, $(x,y)\in G\times G\subset\R^2\times\R^2$.
   }\label{f.graphon}
 \end{figure}

 \subsection{Martingale representation and transport of kernels}
The following martingale representation of $Tf$ is useful both conceptually and computationally.

Given $f\in L^1(K,\nu)$, we generate a martingale sequence $(f_n, n\in \N)$ as follows:
  $$
  f_n\doteq\conE{f}{\cK_n}(x)=\sum_{w\in\Sigma_n} \nu_w(f) \1_{K_w}(x),\quad
  \nu_{w}(f) =\fint_{K_w} f(y) d\nu (y),
  $$
  where $\conE{f}{\cK_n}$ stands for the conditional expectation with respect to the
  $\sigma$-algebra of subsets of $K$ generated by the partition,
  $\cK_n=\sigma\left(\{ K_w,\; w\in\Sigma_n\} \right)$.
  By construction, $\conE{f_{n+1}}{\cK_n}=f_n$. Thus, $(f_n)$ form a martingale sequence converging
  to $f$ $\nu$-a.e. and in $L^1(K,\nu)$ by \cite[Theorem~7.1.3]{Stroock-Prob}.

Next, define $\tilde f_n:Q\to\R$ by
$$
\tilde f_n(x) =\sum_{w\in\Sigma_n} \nu_w(f) \1_{Q_w}(x).
$$
It is shown in \cite{Med2026}, that the sequence $(\tilde f_n)$ converges to
$\tilde f\in L^1(Q,\lambda)$ $\lambda$-a.e. and in $L^1(Q,\lambda)$.
Moreover, $\tilde f=Tf$.  
  
Under the isomorphism constructed above, $f$ and $\tilde f$ generate the same
projections, i.e., $\nu_w(f)=\lambda_w(\tilde f)$ for every $w\in \Sigma^\ast$.

The same construction applies to kernels. Given $W\in L^1(K\times K,\nu\times\nu)$, define
$$
W_{wv}:=\fint_{K_w\times K_v} W\,d(\nu\times\nu),
$$
and set
$$
\tilde W_n(x,y)\doteq \sum_{w,v\in\Sigma_n} W_{wv}\,\1_{Q_w\times Q_v}(x,y).
$$
Then $(\tilde W_n)$ converges in $L^1(Q\times Q,\lambda\times\lambda)$ (and a.e. along a subsequence) to a function $\tilde W$, which is precisely the image of $W$ under the induced isomorphism between $L^1(K\times K,\nu\times\nu)$ and $L^1(Q\times Q,\lambda\times\lambda)$.

\section{Nonlocal dynamics on self-similar sets and Galerkin/IPS approximations}\label{sec.evolution}\label{sec.self-IPS}
\setcounter{equation}{0}

\subsection{A nonlocal evolution problem on $K$}
Let $(K,\{f_i\}_{i=1}^k, \nu)$ be a probabilistic IFS and consider the following nonlocal
evolution equation on $K$
\begin{align}\label{ss-heat}
  \p_tu(t,x)& =F\bigl(t,u(t,x)\bigr)+\int_K W(x,y) D\left(u(t,x),u(t,y)\right)d\nu(y), \\
  \label{ss-heat-ic}
  u(0,x)& =g(x),\qquad x\in K.
\end{align}
We assume $W\in L^2(K\times K, \nu\times\nu)$ and $g\in L^2(K,\nu)$. In addition,
$$
 \bar W\doteq \max\left\{ \esssup_{x\in K} \int_K |W(x,y)| d\nu(y), \;
  \esssup_{y\in K} \int_K |W(x,y)| d\nu(x)\right\}<\infty.
$$
Functions $F(t,u)$ and $D(u,v)$ are jointly continuous and 
\begin{align}\lbl{Lip-f}
  |F(t,u)-F(t,u^\prime)|\le L_F |u-u^\prime|,& \quad \forall t\in\R, \; u, u^\prime\in \R,\\
  \lbl{Lip-D}
  |D(u,v)-D(u^\prime,v^\prime)|\le L_D \left(|u-u^\prime| +|v-v^\prime|\right),& \quad \forall
                                                                                 u,v, u^\prime, v^\prime\in \R,
\end{align}
where $L_F$ and $L_D$ are positive Lipschitz constants.
In addition,
\be\lbl{bound-D}
\sup_{K\times K} |D(u,v)|\le 1.
\ee
Existence and uniqueness of solutions to \eqref{ss-heat}--\eqref{ss-heat-ic} on finite time intervals, with
$u\in C\left([-T, T]; L^2(K,\nu)\right)$, follow from a standard contraction mapping argument; see \cite[Theorem~4.2]{Med2026}.

In many examples one also has an a priori bound of the form
\begin{equation*}
  \sup_{t\in[0,T]}\|u(t,\cdot)\|_{L^\infty(K,\nu)}\le M.
\end{equation*}
In that case, the global assumptions \eqref{Lip-D} and \eqref{bound-D} may be replaced by their local counterparts. More precisely, it is enough to assume that $D$ is locally Lipschitz on $\R^2$ and that for every $M>0$ there exists $C_D(M)$ such that
$|D(u,v)| \le C_D(M)$ whenever $|u|,|v|\le M$.
We shall implicitly use this observation in some of the PDE examples below.

\subsection{Galerkin discretization and deterministic self-similar IPS}\label{sec.Galerkin}
We begin by discretizing the self-similar domain $K$. 
To this end, fix $m \in \mathbb{N}$ and consider the 
self-similar partition $(K_w)_{w \in \Sigma_m}$ of $K$, 
defined by
\begin{equation}\label{partition}
    \mathcal{K}^m\doteq \left\{K_w:\;K_w=f_w(K),\; |w|=m\right\},\qquad
  K=\bigcup_{|w|=m} K_w,
\end{equation}
where we recall that $f_w$ stands for $f_{w_1}\circ f_{w_2} \circ\dots\circ f_{w_m}$.

The partition \eqref{partition} defines a finite-dimensional
subspace of \( \mathcal{H} \doteq L^2(K,\nu) \):
\[
\mathcal{H}^m
:= \operatorname{span}\{\mathbf{1}_A : A \in \mathcal{K}^m\}.
\]
Throughout, \( \mathbf{1}_A \) denotes the indicator function of the set
\( A \).
Although the sets \( K_w \) and \( K_v \) corresponding to distinct
\( w,v \in \Sigma_m \) may intersect, their intersections have
\( \nu \)-measure zero by \eqref{nu-cyl}. Consequently, the resulting
overlaps in the supports of the basis functions do not affect the
convergence properties of the \( L^2 \)-projections onto
\( \mathcal{H}^m \).

We now construct the Galerkin approximation. The solution of the initial
value problem \eqref{ss-heat}–\eqref{ss-heat-ic} is approximated by a
piecewise constant function of the form
\begin{equation}\label{sum}
  u^m(t,x)
  = \sum_{|w|=m} u_w(t)\,\mathbf{1}_{K_w}(x).
\end{equation}
Inserting \eqref{sum} into \eqref{ss-heat} and projecting the resulting
equation onto \( \mathcal{H}^m \) yields 
\begin{equation}\label{re-KM-ss}
  \partial_t u^m(t,x)
  = F\bigl(t,u^m(t,x)\bigr)
  + \int_K W^m(x,y)\,
  D\bigl(u^m(t,x),u^m(t,y)\bigr)\, d\nu(y),
\end{equation}
where the discretized kernel \( W^m \) is given by
\begin{equation}\label{def-Wm}
  W^m
  := \sum_{|w|,|v|=m} W_{wv}\,\mathbf{1}_{K_w \times K_v}, \quad W_{wv}=\fint_{K_w \times K_v}W\,d(\nu\times \nu).
\end{equation}
The initial condition \eqref{ss-heat-ic} is approximated as follows
\begin{equation}\label{approx-g}
  g^m
  := \sum_{|w|=m} g_w\,\mathbf{1}_{K_w},
  \qquad
  g_w := \fint_{K_w} g\, d\nu .
\end{equation}
It is instructive to rewrite the discrete problem \eqref{re-KM-ss}, \eqref{approx-g}
as an initial value problem for a system of ODEs:
\begin{align}\label{ss-KM}
  \dot u_w &= F(t, u_w) + \sum_{|v|=m} W_{wv} D(u_w, u_v)\nu(K_v),\quad
             |w|=m,\\
  \label{ss-KM-ic}
             u_w(0) &= g_w, \quad  |w|=m,
\end{align}
where 
\begin{equation*}
W_{wv}=\fint_{K_w\times K_v} W(x,y) \,d(\nu\times\nu)(x,y), \quad g_w=\fint_{K_w} g\,d\nu.
\end{equation*}
In this formulation, the Galerkin approximation of 
\eqref{ss-heat}--\eqref{ss-heat-ic} admits a natural 
interpretation as an IPS.  $D(u_w,u_v)$ defines the 
pairwise interaction between particles labeled by the words $w,v\in\Sigma_m$, while $W_{wv}$ specifies the 
corresponding interaction weight.

As a system obtained by discretizing a PDE on a self-similar domain, \eqref{ss-KM} inherits the self-similar structure of the underlying continuum model.

\subsection{Random self-similar networks and continuum limit}
Along with the deterministic model \eqref{ss-KM}, we consider an IPS on a $W$-random graph:
\begin{equation}\label{KM-Wss}
  \dot{\mathsf{u}}_w=F(t, \mathsf{u}_w) +\sum_{|v|=n} \xi_{wv} D(\mathsf{u}_w, \mathsf{u}_v)
  \nu(K_v),\quad
  w\in \Sigma_n,
\end{equation}
where $\xi_{wv}$ are independent Bernoulli random variables such that
\begin{equation}\label{ss-Bernoulli}
\fP(\xi_{wv}=1)=\fint_{K_{wv}} W d(\nu\times\nu),\quad \fP(\xi_{wv}=0)=1-\fP(\xi_{wv}=1).
\end{equation}
For the random model, in addition to integrability of $W$ we also assume $W\ge 0$
and $\int_{K\times K} W d(\nu\times\nu)=1$.

Both the deterministic model \eqref{ss-KM} and its
random counterpart \eqref{KM-Wss} converge to the continuum limit \eqref{ss-heat}.
The use of the random graph in \eqref{KM-Wss} may be viewed as a Monte--Carlo approximation of the nonlocal term. This viewpoint is particularly useful numerically, and it suggests sparse random approximations of the lifted kernels as a natural extension; see \cite{KVMed22}.

In sparse variants one should expect a three-way balance between the regularity of
$W$ (or of the lifted kernel), the sampling density of the random graph, and the
Galerkin scale $m$. This is the self-similar analogue of the regularity--sparsity
tradeoff familiar in graphon-based nonlocal dynamics, and it would be natural to treat
it as a separate refinement of the deterministic theory below.

The convergence of the deterministic and random Galerkin schemes is summarized by the following theorem, which is a reformulation of the main continuum-limit estimate proved in \cite{Med2026}.

\begin{theorem}\label{thm.ss-heat}
  Let $u(t,x)$ be the solution of \eqref{ss-heat}--\eqref{ss-heat-ic} with $g\in L^2(K,\nu)$. Likewise, suppose
$$
    u^n(t,x) = \sum_{|j|=n} u_{j}(t) \1_{K_j}(x),\qquad
    \mathsf{u}^n(t,x) = \sum_{|j|=n} \mathsf{u}_{j}(t) \1_{K_j}(x),
$$
  solve the IVPs for \eqref{ss-KM} and \eqref{KM-Wss}, respectively, and satisfy
  $u^n(0,\cdot)=\mathsf{u}^n(0,\cdot)= g^n$.
  Then, for $v\in\{u^n,\mathsf{u}^n\}$,
  \begin{equation}\label{ss-approx-heat}
    \|u-v\|_{C([0,T]; L^2(K,\nu))} \le
    C_T \left( \|g-g^n\|_{L^2(K,\nu)} + \|W-W^n\|_{L^2(K\times K, \nu\times\nu)}\right).
  \end{equation}
  In the random case $v=\mathsf{u}^n$, estimate \eqref{ss-approx-heat} holds almost surely, i.e., for almost every realization of the $W$-random graph \eqref{ss-Bernoulli}.
\end{theorem}

The rate of convergence of the right-hand side in \eqref{ss-approx-heat} can be quantified by taking into account
the regularity of the kernel \( W \) and the initial data \( g \). For Euclidean domains, convergence rates were established
using \( L^p \)-Lipschitz regularity in \cite{KVMed22}. For self-similar domains, analogous estimates were obtained in \cite{Med2026}.
In the self-similar setting, the definition of \( L^p \)-Lipschitz spaces must be suitably adapted to account for the
underlying self-similar structure. 

The randomized model \eqref{KM-Wss} can be further extended to incorporate sparse graphs. From a computational perspective, sparsity can be used to reduce cost without sacrificing accuracy; see \cite{KVMed22} for the Euclidean/graphon counterpart.

\section{Abstract lifting framework and two-scale approximation}\label{sec.prop}
\setcounter{equation}{0}

In this section we formulate the abstract transport-and-discretization procedure
that underlies the rest of the paper. Starting from a reference PDE on $Q$, we
first introduce a nonlocal regularization at scale $\eps$, then transport it to
the self-similar set $K$, and finally discretize the lifted equation on the
level-$m$ partition of $K$. The outcome is a two-parameter approximation scheme
and an explicit self-similar interacting particle system.

\subsection{Transporting a reference PDE from $Q$ to $K$}\label{sec:nonlocalK}

We fix once and for all a self-similar set $K$ equipped with a stationary
self-similar probability measure $\nu$, and, in the framework of
Section \ref{sec.isomorph}, we assume Assumption~\ref{as.small} so that the 
canonical isomorphism of Section \ref{sec.isomorph} applies.
In particular, we use the linear isometry
$$
  T: L^2(K,\nu;\R^p) \longrightarrow L^2(Q,\lambda;\R^p)
$$
introduced in \eqref{def-T}--\eqref{Tf-conj} (and still denoted by $T$ in the
$L^2$-setting), together with its inverse $T^{-1}$. Recall that $Q=[0,1]$.
Here $\lambda$ denotes the measure $\lambda_p=(\pi_Q)_*\mu_p$ introduced in Section~\ref{sec.iso-IFS}; in the general probabilistic IFS setting it is the stationary measure on $Q=[0,1]$ associated with the same probability vector $p$ as on $K$. In particular, if $p=(k^{-1},\dots,k^{-1})$, then $\lambda$ coincides with the Lebesgue measure on $[0,1]$. Since the concrete PDE examples below use standard Sobolev spaces and convolution kernels on $Q$, we henceforth work in this uniform case, so that $\lambda=dx$.

By Theorem~\ref{thm.isomorph}, there exist subsets $Q_0\subset Q$ and
$K_0\subset K$ of full measure and an invertible measure-preserving map
$\Phi:Q_0\to K_0$ such that for every $f\in L^2(K,\nu;\R^p)$ and every
$g\in L^2(Q,\lambda;\R^p)$,
$$
  (Tf)(x)= f\big(\Phi(x)\big)\quad \text{for $\lambda$-a.e. }x\in Q,
  \qquad
  (T^{-1}g)(x)= g\big(\Phi^{-1}(x)\big)\quad \text{for $\nu$-a.e. }x\in K.
$$
We will use $\Phi$ (or $\Phi^{-1}$) only inside functions or kernels, always modulo null sets.

\paragraph{Reference PDE on $Q$.}
We start from an evolution equation on $Q=[0,1]$ of the form
\begin{equation}\label{pdeQ}
  \partial_t u = \bA u, \qquad u(0,\cdot)=u^0(\cdot),
\end{equation}
where $u(t,\cdot)\in L^2(Q,\lambda;\R^p)$ and $\bA$ is an (unbounded) linear
operator with dense domain $\cD(\bA)\subset L^2(Q,\lambda;\R^p)$, as in
\cite{PauTre2025}.
We endow $\cD(\bA)$ with the graph norm
$\|\varphi\|_{\cD(\bA)} := \|\varphi\|_{L^2(Q)} + \|\bA\varphi\|_{L^2(Q)}$.

For the purpose of constructing a corresponding evolution on $K$, it is
convenient to assume that $\bA$ is the generator of a $C_0$-semigroup
$(S(t))_{t\ge 0}$ on $L^2(Q,\lambda;\R^p)$, so that \eqref{pdeQ} is
well-posed in the mild sense, and the solution of \eqref{pdeQ} is $u(t) = S(t)u^0$.

One can treat semilinear/quasilinear perturbations $\partial_t u=\bA u+\cN(u)$
with $\cN$ locally Lipschitz by the usual variation-of-constants approach; we
focus on the linear part for the moment to keep the exposition light.

\paragraph{Transport of the dynamics from $Q$ to $K$.}
Define the transported state $\tilde u$ on $K$ by
\begin{equation}\label{u-tilde-def}
  \tilde u(t,\cdot) := T^{-1}u(t,\cdot) \in L^2(K,\nu;\R^p).
\end{equation}
Then the operator
$$
  \tilde\bA := T^{-1}\,\bA\,T,\qquad \cD(\tilde\bA) := T^{-1}\big(\cD(\bA)\big)
$$
is the generator of a $C_0$-semigroup on $L^2(K,\nu;\R^p)$ given by
$\tilde S(t):=T^{-1}S(t)T$, and $\tilde u$ is the (mild) solution of
\begin{equation}\label{pdeK-abstract}
  \partial_t \tilde u = \tilde\bA\tilde u,\qquad \tilde u(0,\cdot)=\tilde u^0:=T^{-1}u^0.
\end{equation}
This provides a purely operator-theoretic way of transporting the dynamics from $Q$ to $K$.

\begin{remark}\label{rem:intrinsic-vs-transported}
  The evolution \eqref{pdeK-abstract} should be interpreted as the
  \emph{transport} of the $Q$-dynamics through the symbolic/IFS identification.
  In general, $\tilde\bA$ is not the ``intrinsic'' (local) operator one may
  define on $K$ via analysis on fractals (e.g., Kigami's Laplacian, see
  \cite{FukShi92,Kig01,Str06}, which is the self-adjoint
  operator associated with the canonical renormalized energy/Dirichlet form on a
  post critically finite self-similar set---equivalently, the renormalized limit of graph
  Laplacians on prefractal graphs). It is
  instead the natural operator associated with the hierarchical partition
  induced by the IFS and the measure isomorphism $T$.
\end{remark}

\subsection{Nonlocal regularization on $Q$}

To connect \eqref{pdeQ} to interacting particle systems, we insert an
intermediate nonlocal integral approximation.

\paragraph{Integral approximation.}
Following \cite{PauTre2025}, we approximate $\bA$ by a family of integral operators
$$
  \bA_{\eps}u(x) := \int_Q \sigma_{\eps}(x,y)\,u(y)\,dy,\qquad x\in Q,
$$
where $\sigma_{\eps}$ is a kernel depending on $\eps>0$. For simplicity, and because this is the setting directly covered by Theorem~\ref{thm.ss-heat} after lifting to $K$, we often keep the sufficient hypothesis $\sigma_{\eps}\in L^2(Q\times Q)$. Conceptually, however, the approximation is better understood at the operator level: it is enough to require that $\bA_{\eps}\in\cL(L^2)$ and that the associated semigroups satisfy the uniform stability estimate stated below. In boundary-adapted constructions one often obtains this boundedness by Schur's test even when the kernel is not in $L^2$; see \cite{PauTre2025}.

In many PDEs of interest (transport, diffusion, conservation laws),
it is often more convenient to use the ``zero-mass'' form
$$
  \bA_{\eps}u(x) := \int_Q \gamma_{\eps}(x,y)\big(u(y)-u(x)\big)\,dy
$$
(integral terms plus bounded local terms), as will be illustrated in 
the examples treated further.

We denote by $u_{\eps}$ the solution of
\begin{equation}\label{pdeQ-eps}
  \partial_t u_{\eps} = \bA_{\eps}u_{\eps},\qquad u_{\eps}(0)=u^0.
\end{equation}

\paragraph{Stability and consistency requirements.}
In PDE examples, the operator norms $\|\bA_\eps\|_{\cL(L^2)}$ blow up as
$\eps\to0$ (typically like $\eps^{-1}$ for first order terms and $\eps^{-2}$ for diffusion),
because $\bA_\eps$ is a bounded approximation of the unbounded 
operator $\bA$. What is needed for well-posedness and error estimates is a 
\emph{uniform stability} of the associated semigroups. Moreover, the consistency 
estimate usually involves one extra derivative, so we allow a stronger norm.

Let $X$ be a Banach space such that $X\hookrightarrow L^2(Q,\lambda;\R^p)$ densely and
continuously, and $X\subset \cD(\bA)$. We denote its norm by $\|\cdot\|_X$.
The space $X$ is the regularity space in which the consistency error is measured. 
Concretely, it encodes the number of derivatives needed to control the Taylor remainder 
in the nonlocal-to-local approximation. It also ensures that the reference solution of 
\eqref{pdeQ} remains regular enough on $[0,T]$ for the Duhamel estimate leading to 
\eqref{Q-error}. In the examples treated further, $X=H^2(\T)$ is sufficient for first-order 
operators because one remainder derivative is needed, whereas $X=H^4(\T)$ is 
convenient for the heat equation because the second-order approximation leaves a 
remainder involving fourth derivatives.

In view of obtaining a clean two-step limit, we make the following standing assumption, which 
will have to be checked on each application example:

\begin{assumption}\label{ass:Aeps-consistency}
  There exist a dense subspace $\cD_0\subset X$ which is a core for $\bA$ and a function
  $r(\eps)\to0$ as $\eps\to0$ such that:
  \begin{enumerate}
  \item \textbf{Uniform semigroup stability:}
    for any $\eps>0$, $\bA_{\eps}$ generates a $C_0$-semigroup $(S_\eps(t))_{t\ge0}$ on
    $L^2(Q,\lambda;\R^p)$ and there exist constants $M\ge1$ and $\omega\in\R$, independent
    of $\eps$, such that
    $$
      \|S_\eps(t)\|_{\cL(L^2(Q))}\le M e^{\omega t}
      \qquad \forall t\ge0.
    $$
  \item \textbf{Consistency on the core:}
    $$
      \|\bA_{\eps}\varphi-\bA\varphi\|_{L^2(Q)}\le r(\eps)\,\|\varphi\|_{X}
      \qquad \forall\varphi\in\cD_0.
    $$
  \end{enumerate}
\end{assumption}

Under Assumption~\ref{ass:Aeps-consistency}, standard Duhamel and Gronwall arguments yield convergence $u_{\eps}\to u$ for $u^0\in X$ on any finite time interval $[0,T]$ on which the reference solution remains in $X$ (see \cite{PauTre2025}), with an estimate of the form
\begin{equation}\label{Q-error}
  \|u_{\eps}-u\|_{C([0,T];L^2(Q))} \le C_T\,r(\eps)\,\|u^0\|_{X}.
\end{equation}
We will not reproduce the functional analytic proof here; our focus is on
constructing explicit kernels $\sigma_{\eps}$ for representative PDEs, and on
their discretization on $K$.

\subsection{Lifting the nonlocal equation to $K$ and deriving a self-similar IPS}\label{sec_liftK}

\paragraph{Lift of the nonlocal equation.}
Define
$$
  \tilde u_{\eps} := T^{-1}u_{\eps},\qquad \tilde\bA_{\eps} := T^{-1}\bA_{\eps}T.
$$
Then $\tilde u_{\eps}$ is the unique mild solution of
\begin{equation}\label{pdeK-eps}
  \partial_t\tilde u_{\eps} = \tilde\bA_{\eps}\tilde u_{\eps},\qquad \tilde u_{\eps}(0)=\tilde u^0.
\end{equation}
Moreover, by construction, $\tilde\bA_{\eps}$ admits an integral kernel
representation on $K$:
$$
  \tilde\bA_{\eps}\tilde u(x) = \int_K \tilde\sigma_{\eps}(x,y)\,\tilde u(y)\,d\nu(y),
$$
where one convenient choice is
\begin{equation}\label{sigma-lift-explicit}
  \tilde\sigma_{\eps}(x,y):=
  \begin{cases}
    \sigma_{\eps}\big(\Phi^{-1}(x),\Phi^{-1}(y)\big), & (x,y)\in K_0\times K_0,\\
    0, & \text{otherwise.}
  \end{cases}
\end{equation}
With this choice, $\tilde\sigma_{\eps}\in L^2(K\times K,\nu\times\nu)$ and
$\|\tilde\sigma_{\eps}\|_{L^2(K\times K)}=\|\sigma_{\eps}\|_{L^2(Q\times Q)}$.

In the more general Schur-kernel setting mentioned above (where only $\bA_{\eps}\in\cL(L^2)$), 
one should interpret the lifting primarily at the operator level, by defining 
$\tilde\bA_{\eps}:=T^{-1}\bA_{\eps}T$ first and only then averaging it on the cells. 
The explicit $L^2$ kernel formula \eqref{sigma-lift-explicit} is retained here because 
it is exactly the setting covered by Theorem~\ref{thm.ss-heat}.

\paragraph{Galerkin/IPS discretization on the self-similar partition.}
Let $\cK_m=\{K_w:|w|=m\}$ be the level-$m$ partition of $K$ and
$\cH_m:=\mathrm{span}\{\1_{K_w}:|w|=m\}\subset L^2(K,\nu)$ as in
Section~\ref{sec.Galerkin}. Denote by $P_m$ the $L^2$-orthogonal projector onto $\cH_m$.
Define the step-kernel approximation
$$
  \tilde\sigma_{\eps,m}(x,y)
  :=\sum_{|w|,|v|=m} \tilde\sigma^{(\eps)}_{wv}\,\1_{K_w\times K_v}(x,y),
  \qquad
  \tilde\sigma^{(\eps)}_{wv}:=\fint_{K_w\times K_v}\tilde\sigma_{\eps}\,d(\nu\times\nu).
$$
Then the Galerkin approximation of \eqref{pdeK-eps} reads
$$
  \partial_t\tilde u_{\eps,m} = P_m\tilde\bA_{\eps}\tilde u_{\eps,m}
  = \int_K \tilde\sigma_{\eps,m}(\cdot,y)\,\tilde u_{\eps,m}(y)\,d\nu(y),
  \qquad \tilde u_{\eps,m}(0)=P_m\tilde u^0.
$$
Writing
$\tilde u_{\eps,m}(t,x)=\sum_{|w|=m} u^{\eps}_w(t)\,\1_{K_w}(x)$, we obtain the
finite-dimensional ODE system
\begin{equation}\label{IPS-linear}
  \dot u^{\eps}_w(t) = \sum_{|v|=m} \tilde\sigma^{(\eps)}_{wv}\,u^{\eps}_v(t)\,\nu(K_v),
  \qquad |w|=m.
\end{equation}
This is an IPS on the level-$m$ self-similar network.

It is worth emphasizing what \eqref{IPS-linear} represents. 
The unknowns are not attached to Euclidean mesh points, but to the cells $K_w$ 
of the level-$m$ self-similar partition. Thus each degree of freedom $u_w^{\eps}$ is 
the average state carried by one fractal cell, and the coefficients 
$\tilde\sigma^{(\eps)}_{wv}\,\nu(K_v)$ are the effective interaction rates between 
two such cells. In this sense, \eqref{IPS-linear} is simultaneously 
a Galerkin semi-discretization of the lifted nonlocal equation and 
an IPS on a hierarchical network with $N_m=k^m$ particles. 
This is one of the main interests of the construction: the network is not chosen ad hoc, 
but is generated canonically by the IFS itself. Increasing $m$ does not merely refine 
a uniform mesh; it reveals finer generations of the self-similar geometry. 
From the PDE viewpoint, this gives an explicit finite-dimensional approximation 
adapted to the fractal structure. From the network viewpoint, it provides a concrete 
family of multiscale dynamical systems whose continuum limit is a PDE-type evolution 
on $K$. This constructive bridge between the two levels seems to be a genuinely 
new feature of the present framework.

\paragraph{Generalization to nonlinear/nonlocal operators.}
More generally, if $\bA_{\eps}$ is nonlinear but has a nonlocal form
\begin{equation}\label{Aeps-nonlinear}
  (\bA_{\eps}u)(x) = \int_Q W_{\eps}(x,y)\,D\big(u(x),u(y)\big)\,dy,
\end{equation}
then its lift to $K$ yields a kernel $\tilde W_{\eps}$ and the Galerkin scheme
produces precisely the self-similar IPS
\begin{equation}\label{IPS-nonlinear}
  \dot u^{\eps}_w(t) = \sum_{|v|=m} W^{(\eps)}_{wv}\,D\big(u^{\eps}_w(t),u^{\eps}_v(t)\big)\,\nu(K_v),
\end{equation}
which is of the same type as \eqref{ss-KM}. Here,
$$
 W^{(\eps)}_{wv}:=\fint_{K_w\times K_v}\tilde
  W_{\eps}\,d(\nu\times\nu)= \fint_{Q_w\times Q_v} W_{\eps}\,d(\lambda\times\lambda).
$$
\begin{remark}\label{rem:boundedness-technical}
  The assumptions in Section \ref{sec.self-IPS} (in particular, boundedness of $D$ in
  \eqref{bound-D}) are tailored to globally Lipschitz dynamics on a compact
  state space. For the PDE examples below, one can either work in a regime
  where solutions are a priori bounded (e.g., $u^0\in L^\infty$ and maximum
  principles), or insert a bounded saturation or truncation as in
  \cite{PauTre2025} to fit the hypotheses of Theorem~\ref{thm.ss-heat}. 
  A complementary (and useful) improvement would be
  to extend the estimate \eqref{ss-approx-heat} to interaction functions $D$ of
  at most linear growth.
\end{remark}

\paragraph{The two-scale limit $m\to\infty$, $\eps\to0$.}
For fixed $\eps > 0$, the convergence of the Galerkin/IPS approximation
$\tilde u_{\eps,m} \to \tilde u_{\eps}$ as $m\to\infty$ follows from
Theorem~\ref{thm.ss-heat} applied to the kernel $\tilde W_{\eps}$
in the nonlinear case. In the linear case, $\tilde\bA_{\eps}$ is a bounded
operator on $L^2(K,\nu)$ (since $\tilde\sigma_{\eps} \in L^2(K \times K)$),
and the Galerkin convergence is a standard consequence of the approximation
of a bounded operator on the nested subspaces $H_m$; alternatively, it
follows from Theorem~\ref{thm.ss-heat} after truncating the interaction
function on bounded sets, as noted in Remark~\ref{rem:boundedness-technical}.
In both cases:
\begin{equation}\label{K-eps-Galerkin-error}
  \|\tilde u_{\eps,m}-\tilde u_{\eps}\|_{C([0,T];L^2(K,\nu))}
  \le C_T\Big(\|\tilde u^0-P_m\tilde u^0\|_{L^2(K,\nu)} + \|\tilde\sigma_{\eps}-\tilde\sigma_{\eps,m}\|_{L^2(K\times K,\nu\times\nu)}\Big).
\end{equation}
Combining \eqref{Q-error}, \eqref{u-tilde-def} and \eqref{K-eps-Galerkin-error},
we obtain the following result.

\begin{theorem}\label{thm:two-parameter}
Let $u$ solve \eqref{pdeQ}, let $\tilde u=T^{-1}u$, and for each $\eps>0$ let $\tilde u_{\eps,m}$ be the Galerkin/IPS approximation defined above. Under Assumption~\ref{ass:Aeps-consistency}, for every $T>0$,
\begin{equation}\label{global-error}
  \|\tilde u_{\eps,m}-\tilde u\|_{C([0,T];L^2(K,\nu))}
  \le C_T\Big( r(\eps)\,\|u^0\|_{X} + \|\tilde u^0-P_m\tilde u^0\|_{L^2(K,\nu)} + \|\tilde\sigma_{\eps}-\tilde\sigma_{\eps,m}\|_{L^2(K\times K,\nu\times\nu)}\Big).
\end{equation}
\end{theorem}

Thus the derivation of the PDE on $K$ may be viewed as a genuine two-step limit: first $m\to\infty$ at fixed $\eps$ (Galerkin/IPS to the lifted nonlocal equation), then $\eps\to0$ (nonlocal equation to the transported PDE), or any joint limit for which the three terms on the right-hand side of \eqref{global-error} vanish.

The last two terms in \eqref{global-error} are controlled by the approximation theory on $K$ developed
in \cite{Med2026} (in particular, in terms of generalized Lipschitz regularity
relative to the self-similar partition), while the first term is the
``nonlocal-to-local'' error on $Q$ coming from \cite{PauTre2025}.

\paragraph{Balancing the two scales.}
Suppose, in addition, that the Galerkin terms are of order $q^{\alpha m}$, where $q\in(0,1)$ is the common contraction ratio of the IFS and $\alpha>0$ is the generalized Lipschitz exponent from the approximation theory on $K$ developed in \cite{Med2026}. If the nonlocal consistency is $r(\eps)=\cO(\eps^{\beta})$, then \eqref{global-error} becomes
$$
  \|\tilde u_{\eps,m}-\tilde u\|_{C([0,T];L^2(K,\nu))}  \le C_T\big(\eps^{\beta}+q^{\alpha m}\big).
$$
Since the number of cells at level $m$ is $N_m=k^m$ and $q^m=N_m^{-1/s}$ with $s=\log k/\log(1/q)$, the network error is equivalently $N_m^{-\alpha/s}$. Hence the similarity dimension $s$ plays the role of an effective dimension in the cost/error balance. The natural balanced choice is $q^{\alpha m}\asymp \eps^{\beta}$, i.e.
$$
  N_m\asymp \eps^{-\beta s/\alpha},   \qquad\text{or equivalently}\qquad   \eps\asymp N_m^{-\alpha/(\beta s)}.
$$
We thus obtain the following corollary.

\begin{cor}\label{cor:balanced-scales}
Assume, in addition to Theorem~\ref{thm:two-parameter}, that there exist 
$\alpha>0$ and $C>0$, independent of $\eps$ and $m$, such that
$$
  \|\tilde u^0-P_m\tilde u^0\|_{L^2(K,\nu)} +
  \|\tilde\sigma_{\eps}-\tilde\sigma_{\eps,m}\|_{L^2(K\times K,\nu\times\nu)}
  \le C q^{\alpha m},
$$
and that $r(\eps)\le C\eps^{\beta}$ for some $\beta>0$. Then, for every $T>0$,
$$
  \|\tilde u_{\eps,m}-\tilde u\|_{C([0,T];L^2(K,\nu))}
  \le C_T\big(\eps^{\beta}+q^{\alpha m}\big).
$$
If $m=m(\eps)$ is chosen so that $q^{\alpha m(\eps)}\le \eps^{\beta}<q^{\alpha(m(\eps)-1)}$,
then the two contributions are balanced and
$$
  \|\tilde u_{\eps,m(\eps)}-\tilde u\|_{C([0,T];L^2(K,\nu))}
  \le C_T\eps^{\beta}.
$$
Equivalently, since $N_m=k^m$ and $q^m=N_m^{-1/s}$ with $s=\log k/\log(1/q)$, 
the balanced choice $\eps\asymp N_m^{-\alpha/(\beta s)}$ yields
$$
  \|\tilde u_{\eps,m}-\tilde u\|_{C([0,T];L^2(K,\nu))}
  \le C_T N_m^{-\alpha/s}.
$$
\end{cor}

Corollary~\ref{cor:balanced-scales} is the quantitative statement one can use in practice. 
It confirms the interpretation above: once the consistency order $r(\eps)$ is known, 
the estimate prescribes how the nonlocalization scale $\eps$ and the network depth $m$ 
should be coupled. 
For a target accuracy $\delta$, one chooses $\eps\asymp \delta^{1/\beta}$ 
and $N_m\asymp \delta^{-s/\alpha}$; for a fixed particle budget $N_m$, 
the balanced choice is $\eps\asymp N_m^{-\alpha/(\beta s)}$. 
Taking $\eps$ much smaller than this is wasteful, because the network error 
$q^{\alpha m}$ then dominates. 
Taking $m$ much larger is equally wasteful, because the nonlocalization error 
$\eps^{\beta}$ then dominates. 
Thus the rate is not only an asymptotic statement, 
but also a practical design rule for the approximation.

At a conceptual level, Theorem~\ref{thm:two-parameter} is the key quantitative output 
of the construction: it combines, in a single estimate, the nonlocal-to-local approximation 
on the reference space $Q$ and the Galerkin/self-similar network approximation on the  fractal side $K$. 
In the present framework, this appears to give a new constructive and quantitative 
derivation of PDE-type dynamics on fractals from explicit interacting particle systems. 
It also makes visible the role of the similarity dimension $s$: 
it enters the cost/error balance as an effective dimension. 

For the three model problems below one finds $\beta=1$ for transport and Burgers, 
and $\beta=2$ for heat, so diffusion allows a comparatively larger nonlocal scale 
$\eps$ for the same network size and target accuracy.

\section{Representative examples}\label{sec.examples}

In all three examples below we work with periodic boundary conditions on $Q=[0,1]$, i.e., we identify $Q$ with the one-dimensional torus $\T$ and periodize the kernels. The periodic setting isolates the basic mechanism and avoids boundary-layer terms.

The examples should be understood as transported models. In particular, the Burgers and
transport equations below are not meant to coincide with the intrinsic first-order
models built from Dirichlet-form vector fields; compare, for instance, the intrinsic
viscous Burgers and transport/continuity theories in \cite{HinMei20,HinSche24}. 
Their purpose is to show how standard one-dimensional PDEs give rise, after
nonlocalization, lifting and Galerkin projection, to explicit self-similar particle
systems on $K$.

\subsection{Transport equation}

Let $b\in W^{1,\infty}(\T)$. Consider the transport equation on $\T$
\begin{equation}\label{transport}
  \partial_t u + b(x)\,\partial_x u = 0,\qquad u(0)=u^0.
\end{equation}

\paragraph{Nonlocal approximation on $Q$.}
Fix an arbitrary odd function $\eta\in L^1(\R)$ such that $\int_{\R} \eta(z)\,dz = 0$ and 
$\int_{\R} z\,\eta(z)\,dz = -1$ and $C_\eta:=\frac12\int_{\R}|z|^2\,|\eta(z)|\,dz<\infty$.
Define the (periodized) kernel 
\begin{equation}\label{eta-eps-derivative}
  \eta_{\eps}(z):=\eps^{-2}\,\eta(z/\eps)
\end{equation}
and the
nonlocal approximate derivative
\begin{equation}\label{Deps}
  (D_{\eps}u)(x) := \int_{\T} \eta_{\eps}(x-y)\,u(y)\,dy = \int_{\T} \eta_{\eps}(x-y)\,(u(y)-u(x))\,dy .
\end{equation}
Then, for smooth $u$, $D_{\eps}u\to \partial_x u$ in $L^2(\T)$ (more precisely, see \eqref{Deps-consistency} in the proof of Lemma \ref{lemtransport} below).

One obtains an integral approximation of \eqref{transport} by
\begin{equation}\label{transport-eps}
  \partial_t u_{\eps} + b(x)\,D_{\eps}u_{\eps} = 0.
\end{equation}
Equivalently, \eqref{transport-eps} has the form \eqref{pdeQ-eps} with
$\bA_{\eps}u=-b\,D_{\eps}u$, i.e., $\sigma_{\eps}(x,y)=-b(x)\eta_{\eps}(x-y)$.

\begin{lemma}\label{lemtransport}
Assumption~\ref{ass:Aeps-consistency} is satisfied with $M=1$, $\omega=\tfrac12\,\|b'\|_{L^\infty}\int_{\R}|z|\,|\eta(z)|\,dz$, and $r(\eps)=\cO(\eps)$ for the choice $X=H^2(\T)$.
\end{lemma}

\begin{proof}
Here $\bA u=-b\,\partial_x u$ on $L^2(\T)$, with $\cD(\bA)=H^1(\T)$ and a core
$\cD_0=C^{\infty}(\T)$. We take $X=H^2(\T)$.

\smallskip
\noindent\emph{Uniform semigroup stability.}
Although $\|D_\eps\|_{\cL(L^2)}=\|\eta_\eps\|_{L^1}=\eps^{-1}\|\eta\|_{L^1}$ blows up
as $\eps\to0$, the $L^2$-flow of \eqref{transport-eps} is uniformly stable. Since
$\eta_\eps$ is odd, $D_\eps$ is skew-adjoint on $L^2(\T)$. 
Hence, setting $[b,D_\eps]:=bD_\eps-D_\eps(b\cdot)$, for a smooth
solution $u_\eps$ we have
$$
  \langle [b,D_{\eps}]u_{\eps},u_{\eps}\rangle
  = \langle bD_{\eps}u_{\eps},u_{\eps}\rangle-\langle D_{\eps}(bu_{\eps}),u_{\eps}\rangle
  = \langle bD_{\eps}u_{\eps},u_{\eps}\rangle +\langle bu_{\eps},D_{\eps}u_{\eps}\rangle
  =2\langle bD_{\eps}u_{\eps},u_{\eps}\rangle,
$$
and thus
$$
  \frac12\frac{d}{dt}\|u_\eps(t)\|_{L^2(\T)}^2
  =-\langle b D_\eps u_\eps, u_\eps\rangle
  =-\tfrac12\langle [b,D_\eps]u_\eps, u_\eps\rangle,
$$
where $[b,D_\eps]$ has the integral kernel representation
$$
  ([b,D_\eps]u)(x)=\int_{\T} K_{\eps}(x,y)\,u(y)\,dy, \qquad K_{\eps}(x,y):=(b(x)-b(y))\,\eta_{\eps}(x-y).
$$
Since $b\in W^{1,\infty}(\T)$, one has $|b(x)-b(y)|\le \|b'\|_{L^{\infty}(\T)}|x-y|_{\T}$, where $|x-y|_{\T}$ denotes the periodic distance. Therefore
$$
  \sup_x\int_{\T}|K_{\eps}(x,y)|\,dy
  \le \|b'\|_{L^{\infty}(\T)}\int_{\T}|x-y|_{\T}|\eta_{\eps}(x-y)|\,dy
  \le \|b'\|_{L^{\infty}(\T)}\int_{\R}|z|\,|\eta(z)|\,dz,
$$
after the change of variables $x-y=\eps z$. The same estimate holds for $\sup_y\int_{\T}|K_{\eps}(x,y)|\,dx$. Schur's test then gives
\begin{align*}
  \|[b,D_{\eps}]\|_{\cL(L^2(\T))}
  &\le \Big(\sup_x\int|K_{\eps}(x,y)|\,dy\Big)^{1/2} \Big(\sup_y\int|K_{\eps}(x,y)|\,dx\Big)^{1/2} \\
  &\le \|b'\|_{L^\infty(\T)}\int_{\R}|z|\,|\eta(z)|\,dz.
\end{align*}
Consequently,
$$
  \|u_\eps(t)\|_{L^2(\T)}\le \exp\Big(\frac{t}{2}\,\|b'\|_{L^\infty}\int_{\R}|z|\,|\eta(z)|\,dz\Big)\,\|u^0\|_{L^2(\T)},
$$
which is Assumption~\ref{ass:Aeps-consistency}(1) with $M=1$ and $\omega=\tfrac12\,\|b'\|_{L^\infty}\int_{\R}|z|\,|\eta(z)|\,dz$.
Note that $\int_{\R}|z|\,|\eta(z)|\,dz<+\infty$ by the assumptions done on $\eta$. Indeed, split the integral in two parts, $\vert z\vert\leq 1$ and $\vert z\vert>1$; the first is finite because $\eta\in L^1(\R)$, and the second is less than $C_\eta$ because $\vert z\vert\leq\vert z\vert^2$.

\smallskip
\noindent\emph{Consistency.}
After periodization and the change of variables $x-y=\eps z$, one has
$(D_{\eps}u)(x)=\frac1\eps\int_{\R}\eta(z)u(x-\eps z)\,dz$.
Hence
$$
  D_{\eps}u(x)-\partial_xu(x) = \frac1\eps\int_{\R}\eta(z)\Big(u(x-\eps z)-u(x)+\eps z\,\partial_xu(x)\Big)\,dz.
$$
Taylor's formula gives
$$
  u(x-\eps z)-u(x)+\eps z\,\partial_xu(x) = \eps^2 z^2\int_0^1(1-\theta)\,\partial_{xx}u(x-\theta\eps z)\,d\theta.
$$
Therefore, for $u\in H^2(\T)$,
\begin{equation}\label{Deps-consistency}
  \|D_{\eps}u-\partial_xu\|_{L^2(\T)} \le C_\eta\,\eps\,\|\partial_{xx}u\|_{L^2(\T)},
\end{equation}
by Minkowski's inequality and translation invariance of the $L^2(\T)$ norm.

It follows from \eqref{Deps-consistency} that
$\|\bA_\eps u-\bA u\|_{L^2(\T)}\le C_\eta\,\|b\|_{L^\infty}\,\|u\|_{H^2(\T)}\,\eps$.
In particular, Assumption~\ref{ass:Aeps-consistency}(2) holds with $r(\eps)=\cO(\eps)$ for the
choice $X=H^2(\T)$.
\end{proof}

\paragraph{Lift to $K$ and IPS.}
Define $\tilde b:=T^{-1}b\in L^\infty(K,\nu)$ and the lifted kernel
$\tilde\sigma_{\eps}$ by \eqref{sigma-lift-explicit}. The lifted equation is
$$
  \partial_t\tilde u_{\eps}(t,x) + \tilde b(x)\int_K \eta_{\eps}\big(\Phi^{-1}(x)-\Phi^{-1}(y)\big)\,\tilde u_{\eps}(t,y)\,d\nu(y)=0.
$$
The Galerkin scheme yields the particle system
\begin{equation}\label{transport-IPS}
  \dot u^{\eps}_w(t) = -\sum_{|v|=m} b_w\,\eta^{(\eps)}_{wv}\,u^{\eps}_v(t)\,\nu(K_v),
  \qquad b_w:=\fint_{K_w}\tilde b\,d\nu
  = \fint_{Q_w}b\,d\lambda
\end{equation}
with weights
$$
\eta^{(\eps)}_{wv}:=\fint_{K_w\times
  K_v}\eta_{\eps}(\Phi^{-1}(x)-\Phi^{-1}(y))\,d(\nu\times\nu)=
\fint_{Q_w\times
  Q_v}\eta_{\eps}(x-y)\,d(\lambda\times\lambda).
  $$

Let us note a few structural features of \eqref{transport-IPS}. 
First, the unknowns $u_w^{\eps}$ are attached to self-similar cells, 
so the transport dynamics is realized on a hierarchical network rather than on a uniform grid. 
Second, because $\int_{\T}\eta_{\eps}=0$, one has $\bA_{\eps}\1=0$; 
correspondingly, the averaged coefficients satisfy
$$
  \sum_{|v|=m}\eta^{(\eps)}_{wv}\,\nu(K_v)=0,
$$
so constant states are preserved, exactly as for the transport equation. 
Third, the coefficients may have both signs, which is the discrete signature of transport 
rather than diffusion. 
Finally, the scale $\eps$ determines how strongly two cells interact after they are 
compared through the coordinate $\Phi^{-1}$ on $Q$: 
for localized kernels, only cells that are close in the $Q$-parameter interact significantly. 
Thus \eqref{transport-IPS} provides a concrete transport dynamics on a multiscale 
self-similar network, and this seems to be one of the most interesting new features 
of the present approach.

As $\eps\to 0$ with $m$ fixed, the averaged weights $\eta^{(\eps)}_{wv}$ concentrate: since $\eta_\eps$ approximates a distributional derivative, the IPS \eqref{transport-IPS} formally reduces to the Galerkin projection of the local transport equation $\partial_t u + b\,\partial_x u = 0$ onto the level-$m$ partition of $Q$, transported to $K$. In the opposite regime $m\to\infty$ with $\eps$ fixed, the nonlocal structure is preserved but the spatial resolution on $K$ refines. The joint limit $\eps\to 0$, $m\to\infty$ is governed by Theorem~\ref{thm:two-parameter} and Corollary~\ref{cor:balanced-scales}, which prescribe the balanced coupling $\eps\sim k^{-m}$.

\subsection{Burgers equation}
Consider the (inviscid) Burgers equation on $\T$
$$
  \partial_t u + \partial_x\Big(\tfrac12 u^2\Big)=0.
$$
Since solutions may develop shocks, the approximation theory depends on whether
one works with smooth solutions on short times, or with entropy solutions.
Here, a robust strategy is to treat the viscous Burgers
regularization and then discuss the inviscid limit.

\paragraph{A nonlocal conservative flux approximation.}
Let $F(s)=\tfrac12 s^2$ and use the same odd kernel $\eta_{\eps}$ as in
\eqref{Deps}. 
As in the transport case, we use the derivative scaling \eqref{eta-eps-derivative} so that
$D_\eps$ approximates $\partial_x$ (the overall sign can be adjusted by replacing $\eta$
with $-\eta$ if needed).
Consider the integral approximation
$$
  \partial_t u_{\eps} = \bA_{\eps}(u_\eps)
$$
where (recall that $\int_{\T}\eta_\eps=0$)
$$
  (\bA_\eps(u))(x):=-D_\eps(F(u)) = \int_{\T}\eta_\eps(x-y)\big(F(u(x))-F(u(y))\big)\,dy,
$$
which fits \eqref{Aeps-nonlinear} with
$W_{\eps}(x,y)=\eta_{\eps}(x-y)$ and $D(a,b)=F(a)-F(b)$.

\begin{lemma}\label{lem:burgers-smooth-regime}
Fix $R>0$. There exists $C_R>0$, independent of $\eps$, such that for every $u\in H^2(\T)\cap L^\infty(\T)$ with $\|u\|_{H^2(\T)}+\|u\|_{L^\infty(\T)}\le R$,
\begin{equation}\label{burgers-consistency-estimate}
  \|\bA_{\eps}(u)+\partial_xF(u)\|_{L^2(\T)}\le C_R\eps.
\end{equation}
Moreover, if $u_{\eps}$ and $v_{\eps}$ solve $\partial_tu=\bA_{\eps}(u)$ on $[0,T]$ and remain bounded by $R$ in $H^2(\T)$ on this interval, then
\begin{equation}\label{burgers-stability-estimate}
  \|u_{\eps}(t)-v_{\eps}(t)\|_{L^2(\T)}\le e^{C_R t}\|u_{\eps}(0)-v_{\eps}(0)\|_{L^2(\T)},
  \qquad t\in[0,T].
\end{equation}
Consequently, in any smooth regime where the local Burgers solution and the nonlocal approximants remain uniformly bounded in $H^2(\T)\cap L^\infty(\T)$ on $[0,T]$, the nonlinear analogue of Assumption~\ref{ass:Aeps-consistency} is satisfied with rate $r(\eps)=\cO(\eps)$.
\end{lemma}

\begin{proof}
For the consistency estimate, apply \eqref{Deps-consistency} to $F(u)$:
$$
  \|\bA_{\eps}(u)+\partial_xF(u)\|_{L^2(\T)}
  =\|\big(D_{\eps}-\partial_x\big)(F(u))\|_{L^2(\T)}
  \le C_{\eta}\eps\|\partial_{xx}F(u)\|_{L^2(\T)}.
$$
Since $F(s)=\tfrac12 s^2$, one has $\partial_{xx}F(u)=u\,\partial_{xx}u+(\partial_xu)^2$, 
and therefore
$$
  \|\partial_{xx}F(u)\|_{L^2}
  \le \|u\|_{L^\infty}\|u\|_{H^2}+\|\partial_xu\|_{L^4}^2
  \le C\big(1+\|u\|_{L^\infty}\big)\|u\|_{H^2}^2  \le C_R,
$$
using the 1D embeddings $H^1(\T)\hookrightarrow L^\infty(\T)\cap L^4(\T)$. 
This proves \eqref{burgers-consistency-estimate}.

For stability, let $w = u_{\eps} - v_{\eps}$ and $b = \tfrac{1}{2}(u_{\eps} + v_{\eps})$. 
Since $F(u_{\eps}) - F(v_{\eps}) = bw$, the difference equation reads
$\partial_t w + D_{\eps}(bw) = 0$.
Using again the skew-adjointness of $D_{\eps}$, we obtain as in the transport case
$$
  \frac12\frac{d}{dt}\|w(t)\|_{L^2(\T)}^2
  =\langle D_{\eps}(bw),w\rangle
  =-\frac12\langle [b,D_{\eps}]w,w\rangle.
$$
Since $u_{\eps}$ and $v_{\eps}$ are bounded in $H^2(\T)$, the Sobolev embedding 
$H^2(\T)\hookrightarrow W^{1,\infty}(\T)$ yields $\|b'\|_{L^\infty(\T)}\le C_R$. 
The same Schur estimate as in the proof of Lemma \ref{lemtransport} therefore gives 
$\|[b,D_{\eps}]\|_{\cL(L^2(\T))}\le C_R$,
with $C_R$ independent of $\eps$. Hence
$$
  \frac{d}{dt}\|w(t)\|_{L^2(\T)}^2\le C_R\|w(t)\|_{L^2(\T)}^2,
$$
and Gronwall's lemma yields \eqref{burgers-stability-estimate}.
\end{proof}

\paragraph{Lift to $K$ and Galerkin/IPS discretization.}

The corresponding lifted equation on $K$ is
\begin{equation}\label{burgers-eps-K}
  \partial_t\tilde u_{\eps}(t,x)
  + \int_K \eta_{\eps}\big(\Phi^{-1}(x)-\Phi^{-1}(y)\big)\Big(F(\tilde u_{\eps}(t,x))-F(\tilde u_{\eps}(t,y))\Big)\,d\nu(y)=0.
\end{equation}

The particle approximation of \eqref{burgers-eps-K} is
$$
  \dot u^{\eps}_w(t)
  = -\sum_{|v|=m} \eta^{(\eps)}_{wv}\,\Big(F(u^{\eps}_w(t))-F(u^{\eps}_v(t))\Big)\,\nu(K_v),
$$
where $\eta^{(\eps)}_{wv}$ are the same averaged weights as in \eqref{transport-IPS}.

There is also a natural structural interpretation of this Burgers IPS. 
Because the averaged weights inherit the antisymmetry of the odd kernel, 
namely $\eta^{(\eps)}_{wv}=-\eta^{(\eps)}_{vw}$, the scheme is conservative: 
summing the equations over $w$ shows that the discrete mass 
$\sum_{|w|=m}u_w^{\eps}(t)\nu(K_w)$ 
is preserved. 
The nonlinearity acts only through pairwise flux differences $F(u_w^{\eps})-F(u_v^{\eps})$, 
so the model is the hierarchical analogue of a conservative finite-volume discretization. 
In particular, \eqref{burgers-eps-K} and its IPS discretization show that even a 
nonlinear conservation law can be transported to a fractal domain and 
approximated there by an explicit self-similar particle system.

The same limiting regimes as for the transport equation apply: 
as $\eps\to 0$ with $m$ fixed, the nonlocal flux converges formally to the local one, 
while $m\to\infty$ refines the network. The balanced coupling is again 
$\eps\sim k^{-m}$, since the smooth-regime consistency rate is $r(\eps)=\cO(\eps)$.

\paragraph{Entropy regime: two practical routes.}
For inviscid Burgers, the central difficulty is not the particle approximation itself, 
but the selection of the entropy solution after shocks form. 
A first route is therefore to insert a small viscosity and to work with
$$
  \partial_tu+\partial_xF(u)=\kappa\,\partial_{xx}u,  \qquad \kappa>0.
$$
For each fixed $\kappa$, the equation is parabolic, so the smooth-regime estimates above, 
together with the heat approximation of Section~\ref{sec:heat}, apply on finite time intervals. 
One may then justify the IPS approximation for the viscous problem first, 
and only afterwards let $\kappa\to0$. 
By the classical vanishing-viscosity principle, this limit is the entropy solution of Burgers
(see, for instance, \cite{Dafermos2016}). 
The advantage of this route is that entropy selection is delegated to a 
well-understood parabolic regularization.

A second route is to build entropy selection directly into the nonlocal scheme. 
One replaces the centered odd kernel by a one-sided nonnegative kernel and a 
monotone two-point flux. 
Let $\rho\ge 0$ be supported in $[0,\infty)$, with $\int_0^{\infty}\rho(z)\,dz=1$ 
and $\int_0^{\infty}z\rho(z)\,dz=1$, and set $\rho_{\eps}(z):=\eps^{-2}\rho(z/\eps)$. 
For Burgers, define
$$
  g(a,b):=\frac12(a_+)^2+\frac12(b_-)^2,
  \qquad a_+:=\max\{a,0\},\qquad b_-:=\min\{b,0\}.
$$
Then one may consider the one-sided nonlocal operator
$$
  (\cN_{\eps}^{\rm up}u)(x):=\int_0^{\infty}\rho_{\eps}(z)\Big(g(u(x),u(x+z))-g(u(x-z),u(x))\Big)\,dz,
$$
with shifts understood modulo $1$ on $\T$. 
For smooth $u$, a first-order Taylor expansion shows that 
$\cN_{\eps}^{\rm up}(u)\to\partial_xF(u)$. 
After transport to $K$ and Galerkin projection, this yields an IPS with oriented nonnegative 
interaction weights, that is, a nonlocal analogue of an upwind finite-volume scheme. 
This is precisely the structure in which one expects discrete maximum principles, 
$L^1$-contraction, and convergence to the entropy solution; 
compare with the classical role of monotone schemes in 
scalar conservation laws \cite{CrandallMajda1980}.

\subsection{Heat equation}\label{sec:heat}
Let $\kappa>0$. Consider the heat equation on $\T$
$$
  \partial_t u = \kappa\,\partial_{xx}u.
$$

\paragraph{Nonlocal diffusion kernel.}
Let $\rho\ge0$ be an \emph{even} function in $L^1(\R)$ with
$\int_{\R}\rho(z)\,dz=1$ and finite fourth moment $m_4:=\int_{\R}z^4\rho(z)\,dz<\infty$. 
Then $m_2:=\int_{\R} z^2\rho(z)\,dz<\infty$. Set
$\rho_{\eps}(z):=\eps^{-1}\rho(z/\eps)$ (periodized on $\T$) and define
\begin{equation}\label{heat-eps}
  \partial_t u_{\eps}(t,x) = \frac{2\kappa}{m_2\,\eps^2}\int_{\T}\rho_{\eps}(x-y)\big(u_{\eps}(t,y)-u_{\eps}(t,x)\big)\,dy.
\end{equation}
Equation \eqref{heat-eps} is of the form
\eqref{Aeps-nonlinear} with $D(a,b)=b-a$ and
$W_{\eps}(x,y)=\frac{2\kappa}{m_2\,\eps^2}\rho_{\eps}(x-y)$, and
\begin{equation}\label{Aeps_heat-eps}
  (\bA_{\eps}u)(x)=\frac{2\kappa}{m_2\eps^2}\int_{\T}\rho_{\eps}(x-y)\big(u(y)-u(x)\big)\,dy,
\end{equation}
For smooth $u$, the right-hand side of \eqref{heat-eps}
converges to $\kappa\,\partial_{xx}u$ as $\eps\to0$
(see the last part of the lemma below).

\begin{lemma}\label{lem:heat}
Assumption~\ref{ass:Aeps-consistency} is satisfied with $M=1$, $\omega=0$ 
and $r(\eps)=\cO(\eps^2)$ for the choice $X=H^4(\T)$.
\end{lemma}

\begin{proof}
Here $\bA=\kappa\,\partial_{xx}$ on $L^2(\T)$ with $\cD(\bA)=H^2(\T)$ and core
$\cD_0=C^\infty(\T)$. For the consistency estimate we choose $X=H^4(\T)$:

\smallskip
\noindent\emph{Uniform semigroup stability.}
Although $\|\bA_\eps\|_{\cL(L^2)}\sim \eps^{-2}$ because of the prefactor in
\eqref{Aeps_heat-eps}, the operator $\bA_\eps$ is symmetric and dissipative on $L^2(\T)$:
$$
  \langle \bA_\eps u,u\rangle_{L^2(\T)}
  =-\frac{\kappa}{m_2\,\eps^2}\iint_{\T\times\T}\rho_\eps(x-y)\,|u(y)-u(x)|^2\,dx\,dy\le 0.
$$
Hence $\bA_\eps$ generates a contraction semigroup on $L^2(\T)$, uniformly in $\eps$, i.e., Assumption~\ref{ass:Aeps-consistency}(1) holds with $M=1$ and $\omega=0$.

\smallskip
\noindent\emph{Consistency.}
After periodization and the change of variables $x-y=\eps z$, one has
$$
  (\bA_{\eps}u)(x)=\frac{2\kappa}{m_2\eps^2}\int_{\R}\rho(z)\big(u(x-\eps z)-u(x)\big)\,dz.
$$
By Taylor's formula with integral remainder, one has
\begin{multline*}
  u(x-\eps z)-u(x)
  =-\eps z\,\partial_xu(x)+\frac{\eps^2 z^2}{2}\partial_{xx}u(x)-\frac{\eps^3 z^3}{6}\partial_{xxx}u(x) \\
   +\frac{\eps^4 z^4}{6}\int_0^1(1-\theta)^3\partial_{xxxx}u(x-\theta\eps z)\,d\theta.
\end{multline*}
Since $\rho$ is even, the odd moments vanish, while the second moment equals $m_2$. Therefore
$$
  \bA_{\eps}u(x)-\kappa\,\partial_{xx}u(x)
  =\frac{\kappa\eps^2}{3m_2}\int_{\R}z^4\rho(z)\int_0^1(1-\theta)^3\partial_{xxxx}u(x-\theta\eps z)\,d\theta\,dz.
$$
Now use Minkowski's integral inequality together with the translation invariance of the $L^2(\T)$-norm, we get
\begin{align*}
  \|\bA_{\eps}u-\kappa\,\partial_{xx}u\|_{L^2(\T)}
  &\le \frac{\kappa\eps^2}{3m_2}\int_{\R}z^4\rho(z)\int_0^1(1-\theta)^3\|\partial_{xxxx}u(\cdot-\theta\eps z)\|_{L^2(\T)}\,d\theta\,dz \\
  &= \frac{\kappa\eps^2}{3m_2}\Big(\int_{\R}z^4\rho(z)\,dz\Big)\Big(\int_0^1(1-\theta)^3\,d\theta\Big)\|\partial_{xxxx}u\|_{L^2(\T)} \\
  &= \frac{\kappa m_4}{12m_2}\,\eps^2\,\|\partial_{xxxx}u\|_{L^2(\T)}.
\end{align*}
Hence Assumption~\ref{ass:Aeps-consistency}(2) holds with $r(\eps)=\cO(\eps^2)$ and $X=H^4(\T)$.
\end{proof}

\paragraph{Lift to $K$ and IPS.}
The lifted equation on $K$ is
$$
  \partial_t\tilde u_{\eps}(t,x) = \frac{2\kappa}{m_2\,\eps^2}\int_K \rho_{\eps}\big(\Phi^{-1}(x)-\Phi^{-1}(y)\big)\big(\tilde u_{\eps}(t,y)-\tilde u_{\eps}(t,x)\big)\,d\nu(y),
$$
and its Galerkin/IPS approximation reads
\begin{equation}\label{heat-IPS}
  \dot u^{\eps}_w(t)
  = \sum_{|v|=m} \rho^{(\eps)}_{wv}\,\big(u^{\eps}_v(t)-u^{\eps}_w(t)\big)\,\nu(K_v),
\end{equation}
where
$\rho^{(\eps)}_{wv}:=\frac{2\kappa}{m_2\,\eps^2}\fint_{K_w\times K_v}\rho_{\eps}(\Phi^{-1}(x)-\Phi^{-1}(y))\,d(\nu\times\nu)=\frac{2\kappa}{m_2\,\eps^2}\fint_{Q_w\times Q_v}\rho_{\eps}(x-y)\,d(\lambda\times\lambda)$.

The structure of \eqref{heat-IPS} is especially transparent. 
Since $\rho$ is even and nonnegative, the coefficients satisfy $\rho^{(\eps)}_{wv}=\rho^{(\eps)}_{vw}\ge0$. Hence \eqref{heat-IPS} is a \emph{weighted graph Laplacian} on the level-$m$ 
self-similar network. In particular, it preserves the discrete mass 
$\sum_{|w|=m}u_w^{\eps}(t)\nu(K_w)$ 
and dissipates the discrete $L^2$-energy:
$$
  \frac12\frac{d}{dt}\sum_{|w|=m}|u_w^{\eps}(t)|^2\nu(K_w)
  =-\frac12\sum_{|w|,|v|=m}\rho^{(\eps)}_{wv}|u_v^{\eps}(t)-u_w^{\eps}(t)|^2\nu(K_w)\nu(K_v)\le0.
$$
Thus the IPS has the qualitative structure one expects from diffusion: 
symmetric couplings, nonnegative rates, and relaxation toward local equilibrium. 
In the present setting the important point is that these coefficients are not imposed by hand; 
they are obtained by transporting a classical nonlocal diffusion kernel through the IFS coding 
and then averaging it on the self-similar cells.

As $\eps\to 0$ with $m$ fixed, the weights $\rho^{(\eps)}_{wv}$ diverge (reflecting the $\eps^{-2}$ prefactor) but concentrate on pairs of adjacent cells in $Q$; the IPS \eqref{heat-IPS} then formally approaches the Galerkin discretization of $\kappa\,\partial_{xx}u$ on the level-$m$ partition. In the joint limit, Theorem~\ref{thm:two-parameter} with $r(\eps)=\cO(\eps^2)$ prescribes the balanced scaling $\eps\sim k^{-m/2}$.

\section{Intrinsic operators, local charts, and further directions}\label{sec.outlook}
\setcounter{equation}{0}

\subsection{Scope of the construction}

Sections~\ref{sec.prop}--\ref{sec.examples} provide a four-step procedure:
start from a PDE on a standard reference domain $Q$, replace it by a nonlocal
approximation at scale $\eps$, transport that approximation to the fractal
domain $K$ through the IFS isomorphism, and finally discretize the lifted problem
on the level-$m$ partition of $K$. The output is an explicit self-similar IPS,
namely \eqref{IPS-linear} in the linear case and \eqref{IPS-nonlinear} in the
nonlinear/nonlocal case. Theorem~\ref{thm:two-parameter} and
Corollary~\ref{cor:balanced-scales} quantify the full approximation error.

The two-scale estimate is important because the two parameters have different
meanings. The parameter $\eps$ measures the consistency error of the nonlocal
surrogate on $Q$, whereas the level $m$ (or equivalently the number of particles
$N_m=k^m$) measures the resolution of the self-similar network on $K$.
Once the order of consistency $r(\eps)$ is known, the estimate prescribes how the
nonlocalization scale and the network depth should be coupled. For the model
problems treated here, one finds $\beta=1$ for transport and Burgers in the smooth
regime, and $\beta=2$ for heat. In particular, diffusion allows a larger
nonlocalization scale $\eps$ than first-order dynamics for the same network size
and target accuracy.

The operator obtained on $K$ is transported through the symbolic coding, not
intrinsic in the sense of fractal analysis. This is both a limitation and a
feature. It is a limitation if the objective is to recover an operator canonically
attached to the geometry of $K$, such as a Kigami-type Laplacian. It is a feature
if the objective is to model dynamics on hierarchical networks whose natural
continuum description is inherited from a reference PDE posed on a simpler space
$Q$. The present construction should therefore be viewed as complementary to the
intrinsic theory of PDEs on fractals, not as a replacement for it.

There is nevertheless a useful interface with the Dirichlet-form framework. If a lifted
kernel is symmetric and nonnegative, then it defines a quadratic nonlocal energy of the
form
$$
  \cE_\eps(u,u)=\frac12\iint_{K\times K}(u(x)-u(y))^2J_\eps(x,y)\,d\nu(x) \, d\nu(y),
$$
and, under the standard closability and regularity assumptions, a symmetric Dirichlet
form. In that case the calculus of \cite{CipSau03, HinTep13} applies at fixed
$\eps$. The intrinsic question is then a limiting one: after a suitable
renormalization, do the forms $\cE_\eps$ converge, in Mosco or $\Gamma$ sense, to a
local resistance form such as the Kigami energy? Related nonlocal-to-local limits on
Sierpi\'nski-type fractals were studied in \cite{GriYan19,Yang18}. This question is
orthogonal to the two-scale Galerkin estimate proved here, but it provides a natural
bridge between transported nonlocal models and intrinsic fractal diffusion.

\subsection{Nonlinear problems and boundary conditions}

Several ingredients still deserve a systematic treatment before one can claim a
fully general theory. The first is a genuinely nonlinear version of
Theorem~\ref{thm:two-parameter}. For semilinear or quasilinear problems, the
linear semigroup estimate should be replaced by a uniform well-posedness and
Lipschitz-stability statement for the nonlocal flows on bounded subsets of a
regularity space $X$, together with a consistency estimate for the nonlinear
operator on a core. The Burgers example strongly suggests that such an extension
is feasible in the smooth regime, but it should be stated as an abstract theorem
rather than left at the level of representative examples.

A second point is the treatment of boundary conditions on $[0,1]$. The periodic
case is the cleanest one because the kernels are translation invariant and no
boundary layer appears. On a bounded interval, the nonlocal surrogate must encode
the boundary condition through its interaction with the exterior of the domain. A
natural guiding principle is to fix an extension operator $\cE_B$ associated with
the boundary condition $B$, apply the whole-space nonlocal operator to $\cE_B u$,
and then restrict the result back to $[0,1]$.

For homogeneous Dirichlet diffusion, one may take the zero extension $\cE_Du$ and
define
$$
  (\bA_\eps u)(x):=\frac{2\kappa}{m_2\eps^2}\int_{\R}\rho_\eps(x-y)\big(\cE_Du(y)-u(x)\big)\,dy,
$$
that is,
$$
  (\bA_\eps u)(x)=\frac{2\kappa}{m_2\eps^2}\int_0^1 \rho_\eps(x-y)\big(u(y)-u(x)\big)\,dy
  -\frac{2\kappa}{m_2\eps^2}u(x)\int_{\R\setminus[0,1]}\rho_\eps(x-y)\,dy.
$$
This preserves dissipativity and is the natural nonlocal counterpart of the
Dirichlet heat operator. For homogeneous Neumann diffusion, one replaces zero
extension by an even reflection $\cE_Nu$ at the endpoints and sets
$$
  (\bA_\eps u)(x):=\frac{2\kappa}{m_2\eps^2}\int_{\R}\rho_\eps(x-y)\big(\cE_Nu(y)-u(x)\big)\,dy,
$$
which preserves constants and is consistent with $\partial_xu=0$ at the
boundary. For first-order problems, the appropriate mechanism is typically
one-sided: for instance, if $b>0$ and one prescribes an inflow datum at $x=0$,
one should use a kernel supported in $[0,\infty)$ together with an exterior
extension carrying the inflow value on $(-\infty,0)$. This yields a
boundary-adapted nonlocal derivative that approximates $b\partial_xu$ while
enforcing the inflow condition. What remains to be proved in all these cases is
the analogue of Assumption~\ref{ass:Aeps-consistency}: one must establish uniform
semigroup stability (or nonlinear stability) and control the boundary layer
created by the loss of translation invariance near $x=0,1$.

A third point concerns genuinely hyperbolic regimes after shock formation. For
inviscid Burgers, the centered nonlocal approximation is adequate only before
shocks or after viscous regularization. To recover entropy solutions directly,
one needs either a vanishing-viscosity limit or a monotone one-sided nonlocal
flux. Thus entropy admissibility is not an automatic consequence of the present
framework, but an additional structural requirement on the approximation.

\subsection{Local charts on the Sierpinski gasket}

The main geometric gap between the transported framework and intrinsic operators on
fractals comes from adjacency. The global coding of the cylinders $Q_w$ by
subintervals of $[0,1]$ imposes a total order, whereas the cells $K_w$ in a
fractal such as the Sierpinski gasket have a different adjacency graph in
Euclidean space. A possible refinement is therefore to use the interval only
\emph{locally}, as a chart on each self-similar cell, rather than as a single
global coordinate.

For every word $w\in\Sigma_m$, the conjugacy relation
$\Phi\circ g_w=f_w\circ\Phi$ holds almost everywhere on $Q_w$. Hence one may
regard
\begin{equation}\label{local-chart}
  \Phi_w:=f_w\circ \Phi\circ g_w^{-1}:Q_w\longrightarrow K_w
\end{equation}
as a local chart from the interval cell $Q_w$ to the fractal cell $K_w$. In the
case of the Sierpinski gasket, the three first-generation cells $K_1,K_2,K_3$
already form a natural atlas, since their overlaps reduce to boundary points and
therefore have $\nu$-measure zero. For $L^2$-based constructions, this means that
one can work chartwise on each cell and then impose coupling conditions only on
the lower-dimensional interfaces.

More concretely, let $T_w$ denote the local isometry between
$L^2(K_w,\nu|_{K_w})$ and $L^2(Q_w,\lambda|_{Q_w})$ induced by \eqref{local-chart}.
Given a nonlocal approximation $\bA_\eps$ on $Q$, one may define on each cell a
local transported operator
$$
  \tilde\bA_{\eps,w}:=T_w^{-1}\,\bA_\eps\,T_w.
$$
A patchwise diffusion operator can then be assembled from the family
$(\tilde\bA_{\eps,w})_{|w|=m}$ together with compatibility conditions across the
vertices shared by neighboring cells. At least heuristically, this chartwise
construction is closer in spirit to manifold atlases or finite-element assembly
than the global coding used in the main body of the paper, and it may provide a
bridge between the present transported viewpoint and intrinsic operators such as
the Kigami Laplacian \cite{Kig01,Str06}. Making this idea precise would require a
careful analysis of the rescaling, of the matching conditions at common
vertices, and of the limit $m\to\infty$.

A more intrinsic version of this idea would replace symbolic charts by harmonic
coordinates, which are adapted to energy rather than to measure. On the Sierpi\'nski
gasket, harmonic coordinates and the associated Kusuoka-type energy measures provide a
coordinate system closer to the resistance geometry \cite{Kigami08, Teplyaev08}. This is
also consistent with recent works \cite{MedMiz24, MedMiz25} on harmonic maps from p.c.f.
fractals to the circle and on the Kuramoto model on graphs approximating the Sierpi\'nski
gasket, where homotopy classes, harmonic extension, Dirichlet energy and
$\Gamma$-convergence play a central role. These works suggest
that an intrinsic variant of the present paper should probably be formulated
in terms of energy-adapted local charts and prefractal adjacency, rather than through a
single global symbolic coordinate.

\subsection{Pullback of nonlocal equations on the fractal}

The same isomorphism can be used in the reverse direction. Suppose that one starts
from a nonlocal evolution equation on the fractal:
\begin{equation}\label{fractal-nonlocal}
  \partial_t u(t,x)=\int_K J(x,y)\bigl(u(t,y)-u(t,x)\bigr)\,d\nu(y),
\end{equation}
where $J\in L^2(K\times K,\nu\times\nu)$ (or, more generally, satisfies suitable
Schur-type bounds). Define the pullback kernel
\begin{equation}\label{pullback-kernel}
  J^\Phi(\xi,\eta):=
  \begin{cases}
    J\bigl(\Phi(\xi),\Phi(\eta)\bigr), & (\xi,\eta)\in Q_0\times Q_0,\\
    0, & \text{otherwise.}
  \end{cases}
\end{equation}
Then the dynamics on $K$ is conjugate to a nonlocal equation on $Q$.
The following statement then follows immediately from the measure-preserving 
property of $\Phi$ and the definitions of $T$ and $J^\Phi$.

\begin{proposition}\label{prop:pullback-nonlocal}
Let $\cL_K$ be the operator defined by the right-hand side of
\eqref{fractal-nonlocal}, and let $\cL_Q$ be the operator on $L^2(Q,\lambda)$
defined by
$$
  (\cL_Q v)(\xi)=\int_Q J^\Phi(\xi,\eta)\bigl(v(\eta)-v(\xi)\bigr)\,d\lambda(\eta).
$$
Then $T\,\cL_K\,T^{-1}=\cL_Q$.
Consequently, $u$ solves \eqref{fractal-nonlocal} on $K$ if and only if
$v:=Tu$ solves $\partial_t v = \cL_Q v$ on $Q$.
\end{proposition}

This simple observation suggests a complementary numerical program. Instead of
transporting a classical PDE from $Q$ to $K$, one may start from a native
nonlocal equation on the fractal and pull it back to the interval, where one can
apply one-dimensional quadrature, sparse approximation, or Monte--Carlo methods.
A particularly attractive case is the fractional heat equation
$\partial_t u + (-\Delta_K)^s u = 0$, for $0<s<1$,
on the Sierpinski gasket, where $\Delta_K$ denotes an intrinsic Laplacian on the
fractal and $(-\Delta_K)^s$ is defined through spectral calculus. Whenever such
an operator admits an integral representation (either directly as a singular kernel
or through a fractional Riesz kernel) its pullback through \eqref{pullback-kernel}
produces a singular one-dimensional nonlocal operator on $Q$. Recent work on
fractional Gaussian fields and fractional Riesz kernels on the Sierpinski gasket
indicates that this direction is analytically meaningful \cite{BauLac22}.
Developing numerical schemes for these pulled-back operators would naturally
complement the transport framework of the present paper.

This reverse use of the isomorphism is close in spirit to recent integral-equation
methods for acoustic scattering by fractals, where the physical model is first written
as a nonlocal integral equation on the fractal scatterer and then discretized by
Galerkin and quadrature methods \cite{CaetanoEtAl2025}. Our contribution in such a
setting would be different: the symbolic isomorphism gives a systematic way of pulling
native fractal integral equations back to a one-dimensional reference space, where
classical quadrature, sparse approximation, randomized sampling, or the self-similar
cell averages developed in this paper can be applied.

\subsection{Additional directions}
The previous subsections point to several further questions. One may compare the
transported operators obtained here with intrinsic operators on specific
fractals, study sparse or random self-similar networks in the spirit of
\cite{KVMed22}, or consider higher-dimensional reference spaces $Q$ and
anisotropic kernels. From the particle-system viewpoint, it would also be
interesting to understand systematically which qualitative properties of the
reference PDE survive after transport and Galerkin discretization: dissipation
for diffusion, conservative or one-sided structures for transport and scalar
conservation laws, maximum principles, entropy inequalities, or synchronization
thresholds in oscillator-type models. Clarifying these correspondences would
significantly strengthen the conceptual content of the framework.

Another promising direction is optimal transport on fractals. The map $\Phi$ identifies
probability measures on $K$ that are absolutely continuous with respect to $\nu$ with
probability measures on $Q$ that are absolutely continuous with respect to $\lambda$.
If the cost is defined through the symbolic coordinate, the problem is reduced to a
classical one-dimensional optimal-transport problem; if the cost is intrinsic on $K$
(for example Euclidean or resistance distance), its pullback becomes a nonlocal cost on
$Q$. This gives a flexible framework in which one can separate measure-theoretic
transport, intrinsic geometry, and numerical discretization.

\medskip
\noindent {\bf Acknowledgements.}
This work was partially supported by the National Science Foundation through
grant DMS-2406941 (to GSM).

\small
\bibliographystyle{abbrv}
\bibliography{galerkinSG,galerkinSG_additions}

\end{document}